\documentclass[10pt]{article}
\usepackage[nohead,margin=1.0in]{geometry}
\usepackage{amssymb, amsmath, amsthm, amsfonts}
\usepackage{graphicx,epsfig}
\usepackage{cite}
\usepackage{textcomp}
\usepackage{times}
\usepackage{float}
\graphicspath{{./pics/}}
\usepackage{color}
\usepackage{epstopdf}
\usepackage{mathrsfs}
\usepackage{subcaption}
\usepackage{threeparttable}
\usepackage{algorithm,algorithmic}
\usepackage{booktabs}
\usepackage{verbatim}
\usepackage[colorlinks,linktocpage,linkcolor=blue]{hyperref}
\usepackage{bm}
\numberwithin{equation}{section}
\newtheorem{theorem}{Theorem}[section]
\newtheorem{remark}{Remark}[section]
\newtheorem{proposition}{Proposition}[section]
\newtheorem{lemma}{Lemma}[section]

\newtheorem{definition}[theorem]{Definition}
\newtheorem{condition}[theorem]{Condition}
\newtheorem{assumption}[theorem]{Assumption}
\newtheorem{example}{Example}[section]

\allowdisplaybreaks

\def\II{(\Omega)}
\def\rbd{{\rm\bf d}}
\def\diam{{\rm diam}}
\def\dist{{\rm dist}}

\title{Error Analysis of the Inverse Conductivity Problem with Scattered Measurements\thanks{B. Jin is supported by Hong Kong RGC General Research Fund (14306824) and ANR / Hong Kong RGC Joint
Research Scheme (A-CUHK402/24) and a start-up fund from The Chinese University of Hong Kong. The work of Q. Quan is supported by UM Postdoctoral Fellow scheme from University of Macau. W. Zhang is partially supported by the National Natural Science Foundation of China under grant numbers No.12371423 and No.12241104, Guangdong Provincial Key Laboratory of Computational Science and Material Design (No.2019B030301001).}}

\author{Bangti Jin\thanks{Department of Mathematics, The Chinese University of Hong Kong, Shatin, New Territories, Hong Kong, P. R. China (\texttt{bangti.jin@gmail.com, b.jin@cuhk.edu.hk}).} \and Qimeng Quan\thanks{State Key Laboratory of Internet of Things for Smart City, University of Macau, Avenida da Universidade, Taipa, Macau, P. R. China. (qimengquan@um.edu.mo, quanqm@whu.edu.cn)} \and Wenlong Zhang\thanks{Department of Mathematics, Southern University of Science and Technology (SUSTech), Shenzhen, Guangdong, China. Guangdong Provincial Key Laboratory of Computational Science and Material Design, Southern University of Science and Technology. (zhangwl@sustech.edu.cn)}}

\begin{document}

%\subjclass{...}
\maketitle

\begin{abstract}
{In this work, we investigate the inverse problem of recovering the conductivity coefficient in an elliptic equation from noisy measurements collected at finitely many deterministic scattered points in the domain $\Omega$, and corrupted by random noise. Inspired by the regularity analysis, we propose a numerical scheme based on the regularized least-squares formulation with a $W^{1,4}(\Omega)$ penalty, and discretize the regularized problem using the Galerkin finite element method with continuous piecewise linear elements. Under suitable assumptions on the problem data, we provide an error analysis of the regularized solution and its Galerkin approximation. We establish $L^2(\Omega)$ error bounds in a high-probability sense, which depend explicitly on the regularization parameter $\gamma$, the number $n$ of data points and the mesh size $h$. We also present numerical experiments to illustrate the theoretical findings.}
\end{abstract}

\section{Introduction}
Let $\Omega\subset \mathbb{R}^d\ (d=2,3)$ be a convex bounded polygon / polyhedron with a boundary $\partial\Omega$. Consider the following boundary value problem: 
\begin{equation}\label{eqn:gov}
   \left\{\begin{aligned}
    -\nabla\!\cdot\!(q\nabla u) & = f, \ &\mbox{in}&\ \Omega, \\
    u&=0, \ &\mbox{on}&\ \partial\Omega,
	\end{aligned}
	\right.
\end{equation}
where the function $f\in L^2(\Omega)$ is a given source. The conductivity $q$ belongs to the admissible set 
\begin{equation*}
	\mathcal{A} = \{q\in L^\infty\II: c_0\leq q(x)\leq c_1 \mbox{ a.e. in } \Omega\},
\end{equation*}
with $0< c_0<c_1<\infty$. We indicate by $u(q)$ the dependence of the solution $u$ to problem \eqref{eqn:gov} on the conductivity $q$. The model \eqref{eqn:gov} arises in several important physical processes, e.g., groundwater flow, including the hydraulic diffusivity in a confined inhomogeneous aquifer \cite{frind1973galerkin,yeh1986review}. 
The inverse conductivity problem (ICP) is to estimate the conductivity $q^\dag\in \mathcal{A}$ from finitely  many noisy pointwise measurements $\bm m :=(m_i)_{i}^n$, collected at deterministic locations $(x_i)_{i=1}^n\subset\Omega$ (see Assumption \ref{ass:quasi-points} for more details) and corrupted by additive noise:
\begin{equation}\label{eqn:data}
   m_i = u(q^\dagger) (x_i) + \xi_i,\quad i = 1,\dots,n.
\end{equation}
The random noise $\bm\xi=(\xi_i)_{i=1}^n$ follows independent and identically distributed (i.i.d.) sub-Gaussian distributions with zero mean $\mathbb{E}[\bm \xi ]= \bm 0$ and noise level $\sigma$; see the appendix for the definition of sub-Gaussian random variables.
In practice, the measurement data $\bm{m}$ are typically acquired using a finite number of sensors of limited precision at fixed locations. Thus the model \eqref{eqn:data} reflects practically realistic physical constraints.

Due to the inherent ill-posedness, suitable regularization strategies are required for the stable reconstruction of the conductivity $q$. One common strategy is to first construct a regularized functional in the framework of variational regularization \cite{Engl:1996, ItoJin:2015}, and then to approximate the functional using suitable discrete schemes, e.g., finite element method (FEM).
There has been significant progress in the analysis of the ICP and its numerical approximations, including stability estimate \cite{Alessandrini:1986,Bonito:2017} and convergence (rates) of the continuous and discrete schemes \cite{zou:1998,keung1998numerical,Hao:2010}.  Hao et al \cite{Hao:2010} derived the convergence rate on the regularized solution using the source condition. In the case of a Neumann boundary condition, \cite{Falk:1983} approximated the output least-squares formulation using the Galerkin FEM, and derived error bounds for the discrete reconstruction. The analysis is based on first taking a suitable test function in the weak formulation to derive a weighted error estimate, and then removing the weight function using the non-zero gradient condition \cite[p. 539]{Falk:1983}. See also \cite{Wang:2010} for relevant results in the regularized case. For problem \eqref{eqn:gov},  \cite{JinZhou:2021SINUM}  investigated P1 elements for  approximating the regularized problem, and derived error estimates on the discrete approximation. With a proper choice of the mesh size $h$ and the regularized parameter $\gamma$, the error bound depends sublinearly on the noise level. The key in the analysis in \cite{Falk:1983,Wang:2010,JinZhou:2021SINUM} is to combine conditional stability with the FEM analysis; See \cite{Cen:2025} for an overview.
These works rest on two key premises: the measurement noise is deterministic and of small magnitude, and the observational data is available almost everywhere in the whole domain $\Omega$. However, practical scenarios modeled by the scheme \eqref{eqn:data} violate these two key assumptions, and the analysis therein does not apply directly. 

So far the convergence analysis of the measurement scheme \eqref{eqn:data} remains rather limited in the context of PDE parameter identification \cite{chen2022stochastic,SunZhang:2025sinum,jin2024stochastic}.  Chen et al \cite{chen2022stochastic} investigated a Galerkin FEM scheme for identifying a spatially dependent source in a parabolic model, with pointwise random measurements at the terminal time, and established error bounds in weaker topologies for both continuous and discrete solutions. The key tools in  \cite{chen2022stochastic} include the spectral decomposition of the linear forward operator and noise separation technique in the variational formulation of the state. See also \cite{SunZhang:2025sinum} for  recovering a space-dependent source in a parabolic equation from boundary measurements.  \cite{jin2024stochastic} investigated the inverse potential problem for a second-order elliptic equation with scattered measurements, and under the $H^1(\Omega)$ penalty, derived error bounds on the continuous and discrete regularized approximation. In the ICP, since $q$ enters into the leading term of the elliptic operator, the solution operator has very limited smoothing property for low-regularity conductivity $q$. Thus the standard $H^1(\Omega)$ penalty does not appear sufficient for the analysis of the regularized problem.

In this work, we develop a reconstruction scheme for the ICP based on a $W^{1,4}(\Omega)$ penalty on the conductivity $q$ and discretize the regularized problem using the Galerkin FEM. First, we establish the well-posedness of the regularized problem and derive the following Lipschitz stability estimate respectively in Proposition \ref{prop:con-wellposedness} and Lemma \ref{lem:reg-u}:
\begin{equation}\label{intro-Lip-estimate}
    \|u(q)-u(\widetilde{q})\|_{H^2(\Omega)}\leq c\|\widetilde{q}-q\|_{W^{1,4}\II}\|\widetilde q\|_{W^{1,4}\II}^\zeta\|q\|_{W^{1,4}\II}^\zeta, \quad \forall q, \widetilde{q}\in W^{1,4}(\Omega)\cap \mathcal{A},
\end{equation}
where the exponent $\zeta=2$ for $d=2$ and $\zeta = 4$ for $d=3$. Second, in Theorem \ref{thm:con-err-conduc}, we provide an error bound for the regularized solution $q^*$. This is achieved by first establishing high‑probability bounds for the regularized solution $q^*$ in the $W^{1,4}(\Omega)$ norm and for the state approximation $u(q^\dag)-u(q^*)$ in the semi‑discrete norm $\|\cdot\|_n$ in Lemma \ref{lem:apprx-u-bound-q}:
\begin{equation}\label{intro-apprx-u-bound-q}
\mathbb{P}(\|u(q^*)-u^\dagger\|_n\geq \gamma\rho_0^{\frac\kappa2}z) < 2\exp(-cz^2)\quad\mbox{and}\quad\ \mathbb{P}(\| q^*\|_{W^{1,4}(\Omega)}\geq \rho_0z) < 2\exp(-cz^2),\ \forall z>0,
\end{equation}
where $\mathbb{P}$ denotes the law of the data $\bm{m}$, $\kappa = 4\zeta$, $\rho_0$ is defined in \eqref{eqn:rho0-gamma}, and $\gamma>0$ is the regularized parameter, and then combining these bounds with the conditional stability estimate in \cite[Theorem 3.2]{Bonito:2017}. The estimate \eqref{intro-Lip-estimate} plays a crucial role in proving Lemma \ref{lem:apprx-u-bound-q}. Third and last, we analyze the discrete regularized problem (using the Galerkin FEM with continuous piecewise linear elements), and  establish a high-probability \textit{a priori} stochastic $L^2(\Omega)$ error bound in Theorem \ref{thm:err-conduc}. Since the discrete state $u_h(q_h^*)$ lacks sufficient Sobolev regularity, the argument in the discrete case differs substantially from that in the regularized formulation.

The rest of the paper is organized as follows. In Section \ref{sec:main-diss}, we describe the regularized formulation and the main theoretical results. In Sections \ref{sec:conv-cont} and \ref{sec:conv-disc}, we provide the proofs for Theorems \ref{thm:con-err-conduc} and \ref{thm:err-conduc}, respectively. In Section \ref{sec:numer}, we present numerical illustrations. In Appendix \ref{sec:pre}, we review preliminaries on sub-Gaussian random variables. For any $m\geq0$ and $p\geq1$, we denote by $W^{m,p}\II$ the standard Sobolev spaces of order $m$, equipped with the norm $\|\cdot\|_{W^{m,p}\II}$ and also write $H^{m}\II$ with the norm $\|\cdot\|_{H^m\II}$ when $p=2$ \cite{Adams2003Sobolev}. We denote by $(\cdot,\cdot)$ the $L^2\II$ inner product, and by $C(\overline{\Omega})$ the space of continuous functions on $\overline{\Omega}$. Throughout, we denote by $c$ a generic positive constant not necessarily the same at each occurrence but always independent of the mesh size $h$, the number $n$ of observation points $(x_i)_{i=1}^n$ and the penalty parameter $\gamma$. Let $\zeta = 2\chi_{d=2} + 4\chi_{d=3}$ and $\kappa = 4\zeta$, with $\chi_E$ being the characteristic function  of the set $E$.

\section{Main results and discussions}
\label{sec:main-diss}
\subsection{Regularized problem and FEM approximation}

To recover the conductivity $q$ from the discrete noisy data $\bm m$, we
employ the least-squares scheme with a $W^{1,4}\II$ penalty \cite{Engl:1996,ItoJin:2015}, by minimizing the following regularized functional:
\begin{equation}\label{eqn:conti-optim-prob-conduc}
	\min_{q\in \mathcal{A}} J_\gamma(q)=\big\|u(q) - \bm m\big\|_n^2 +\gamma\| q\|_{W^{1,4}\II}^\kappa,
\end{equation}
where the scalar $\gamma>0$ is the regularization parameter and the state $u\equiv u(q)\in H^1_0\II$ satisfies
\begin{equation}\label{eqn:conti-weak-conduc}	
	\big(q\nabla u,\nabla\varphi\big)=\big(f,\varphi\big), \quad\forall\varphi\in H^1_0\II.
\end{equation}
The discrete semi-norm is given by $\|v\|_n:=\sqrt{(v,v)_n}$ for every $v\in C(\overline{\Omega})$, with the discrete points $(x_i)_{i=1}^n\subset\Omega$ and for $ y\in\mathbb{R}^n$ and $v,w\in C(\overline{\Omega})$:
\begin{align*}
	(y,v)_n & = n^{-1}\sum^n_{i=1}y_iv(x_i)\quad\mbox{and}\quad
	(v,w)_n  = n^{-1}\sum^n_{i=1}v(x_i)w(x_i).
\end{align*}
Problem \eqref{eqn:conti-optim-prob-conduc}--\eqref{eqn:conti-weak-conduc} has at least one global minimizer $q^*\in \mathcal{A}\cap W^{1,4}(\Omega)$ $\mathbb{P}$ almost surely ($\mathbb{P}$ denotes the law of the data $\bm m$). Moreover, the minimizer $q^*$ depends continuously on the data  $\bm m$, cf. Proposition \ref{prop:con-wellposedness}. The choice of the $W^{1,4}(\Omega)$ penalty is motivated by the theoretical analysis: The choice allows deriving crucial analytic properties of the forward map, e.g., regularity esetimate and Lipschitz stability; see Lemma \ref{lem:reg-u} for the details. These properties play a crucial role in the analysis of both regularized solution and its discrete approximations. The choice of the exponent $\kappa$ is also crucial in the error analysis.

In practice, we employ the Galerkin FEM to discretize problem \eqref{eqn:conti-optim-prob-conduc}-\eqref{eqn:conti-weak-conduc}. Let $h\in(0,h_0]$ for some $h_0<1$ and  $\mathcal{T}_h:=\cup\{K_j\}_{j=1}^{N_h}$ be a quasi-uniform simplicial triangulation of the domain $\Omega$ into mutually disjoint open face-to-face
subdomains $K_j$, such that $\Omega:= {\rm Int}(\cup_j\{\overline{K}_j\})$ \cite[Section 5.3]{LarssonThomee:2003}. On each element $K\in\mathcal{T}_h$, we denote by $P_r(K)$ the space of polynomials of degree at most $r$ over $K$. Next we define the $H^1(\Omega)$ conforming space (P1 elements) by
\begin{equation*}
	V_{h}:=\{v_{h}\in H^1(\Omega):v_h|_{K}\in P_1(K),\, \forall K\in\mathcal{T}_h\} \quad \mbox{and}\quad X_h:= V_h\cap H_0^1\II.
\end{equation*}
The space $X_h$ and the set $\mathcal{A}_h:= V_h\cap\mathcal{A}$ are used to discretize the state $u$ and conductivity $q$, respectively. 

Now we formulate the Galerkin approximation of problem \eqref{eqn:conti-optim-prob-conduc}--\eqref{eqn:conti-weak-conduc} by
\begin{equation}\label{eqn:dis-optim-dis}
\min_{q_h\in \mathcal{A}_h} J_{\gamma,h}(q_h)=\big\|u_h(q_h)-\bm m\big\|_n^2 + \gamma\| q_h\|_{W^{1,4}\II}^\kappa,
\end{equation}
where the discrete state $u_h\equiv u_h(q_h)\in X_h$ satisfies
\begin{equation}\label{eqn:dis-weak-conduc}	
	\big(q_h\nabla u_h,\nabla\varphi_h\big)=\big(P_hf, \varphi_h\big), \quad\forall\varphi_h\in X_h.
\end{equation}
By the norm equivalence and compactness in finite-dimensional spaces and the argument of Proposition \ref{prop:con-wellposedness}, we deduce that there exists at least one global minimizer $q_h^*\in V_h$ to problem \eqref{eqn:dis-optim-dis}-\eqref{eqn:dis-weak-conduc} $\mathbb{P}$ almost surely.

\subsection{Error estimates}
Now we give $L^2\II$ error bounds on the regularized solution $q^*$ and the Galerkin approximation $q_h^*$. 
Throughout we assume a quasi-uniform distribution of the observational points $(x_i)_{i=1}^n$ over the domain $\Omega$. Let
$d_{\max}= {\rm sup}_{x\in \Omega} \mathop {\rm inf}_{1 \leq i \leq n} \|x-x_i\|_{\ell^2}$ and $ 
d_{\min}={\rm inf}_{1 \leq i \neq j \leq n} \|x_i-x_j\|_{\ell^2}$,
with $\|\cdot\|_{\ell^2}$ being the standard Euclidean norm of vectors. Assumption \ref{ass:quasi-points} is commonly used in  statistical regularization  \cite{chen2018stochastic,chen2022stochastic}. Roughly speaking, Assumption \ref{ass:quasi-points} ensures that the norms $\|\cdot\|_n$ and $L^2(\Omega)$ are comparable.

\begin{assumption}\label{ass:quasi-points}
The sampling points $(x_i)_{i=1}^n$ are scattered but quasi-uniformly distributed in $\Omega$: there exists a constant $b>0$ independent of $n$ such that $d_{\max} \leq bd_{\min}$ is valid for all large $n$. 
\end{assumption}

For the analysis of the regularized solution $q^*$, we make the following regularity assumption on $q^\dag$ and $f$.  
\begin{assumption}\label{ass:con-reg}
$q^\dagger\in W^{1,4}\II\cap\mathcal{A}$ and $f\in L^2\II$.
\end{assumption}

The next condition plays a crucial role in deriving an $L^2\II$ error bound. By \cite[Lemmas 3.3 and 3.7]{Bonito:2017}, it holds with $\beta=2$, if $\Omega$ is Lipschitz, $q^\dag\in\mathcal{A}$ and $f\in L^2\II$ with $f\geq c_f>0$, and $\beta=0$ if  $\Omega$ is $C^{2,\alpha}$, $q^\dag\in C^{1,\alpha}(\overline{\Omega})\cap\mathcal{A}$ and $f\in C^{0,\alpha}\II$ with $f\geq c_f>0$. The exponent $\beta\geq0$ quantifies the decay rate of the weight near $\partial\Omega$. The precise conditions on the problem data for ensuring the condition with $\beta\in (0,2)$ appear unknown so far.
\begin{condition}\label{Cond:positivity}
There exists $\beta\geq0$ such that $(fu^\dag+q^\dag |\nabla u^\dag|^2)(x)\geq c~\dist(x,\partial\Omega)^\beta~\mbox{a.e.}~ x\in \Omega$.
\end{condition} 

Now we can state the first main theorem on the regularized solution $q^*$.
\begin{theorem}\label{thm:con-err-conduc}
Let Assumptions \ref{ass:quasi-points} and \ref{ass:con-reg} hold, and let $q^*\in\mathcal{A}$ be a minimizer to problem \eqref{eqn:conti-optim-prob-conduc}-\eqref{eqn:conti-weak-conduc}. Fix $\tau\in(0,\frac14)$,
\begin{equation}\label{eqn:rho0-gamma}
\rho_0 = \|q^\dagger\|_{W^{1,4}\II} + \sigma n^{-\frac12}\quad\mbox{and} \quad \gamma^{\frac12 + \frac d8}= O(\sigma n^{-\frac12}\rho_0^s) \quad\mbox{with } s = \tfrac{-2\kappa+d}{4}.    
\end{equation}
Then with probability at least $1-2\tau$, with $\ell_\tau = \sqrt{\log\frac{2}{\tau}}$, there exists $c$, independent of $q^*$, $\gamma$, $\rho_0$, $\tau$ and $n$, such that
\begin{equation*}
\int_{\Omega}\Big(\frac{q^\dag - q^*}{q^\dag}\Big)^2(fu^\dag + q^\dag|\nabla u^\dag|^2)\ {\rm d}x \leq c\big(\gamma^{\frac12} + n^{-\frac2d}\big)^{\frac12}(\rho_0\ell_\tau)^{2\zeta+2}.
\end{equation*}
Moreover, if Condition \ref{Cond:positivity} holds, then
\begin{equation*}
	\|q^{\dag}-q^*\|_{L^2\II}\leq c\big[\big(\gamma^{\frac12} + n^{-\frac2d}\big)^{\frac12}(\rho_0\ell_\tau)^{2\zeta+2}\big]^{\frac{1}{2(1+\beta)}}.
\end{equation*}
\end{theorem}
\begin{remark}\label{rem:con-rate}
Note that in \eqref{eqn:data}, the noise level at each point is fixed at $\sigma$, and the number $n$ of observation points can vary. The effective noise level roughly scales like $\frac{\sigma}{\sqrt{n}}$, when compared with the deterministic setting \cite{Engl:1996,ItoJin:2015}. The analysis thus focuses on the dependence of the error estimates with respect to $n$.
Theorem \ref{thm:con-err-conduc} gives an $L^2\II$ error bound on the regularized solution $q^*$ in a high probability sense. If the number $n$ of sampling points is large enough $($i.e., $\gamma\geq n^{-\frac4d}$$)$, then
    \begin{equation*}
     \|q^{\dag}-q^*\|_{L^2\II}\leq c\gamma^{\frac{1}{8(1+\beta)}} (\rho_0\ell_\tau)^{\frac{1+\zeta}{1+\beta}}.   
    \end{equation*}
    This estimate is consistent with the conditional stability in \cite[Theorem 3.2]{Bonito:2017} and the $L^2\II$ error bound of the deterministic model \cite[Corollary 3.3]{JinZhou:2021SINUM}. 
\end{remark}
Next, we analyze the discrete approximation $q_h^*$ under the following assumption. 
\begin{assumption}\label{ass:dis-reg}
    $q^\dagger\in H^2\II\cap W^{1,\infty}\II\cap\mathcal{A}$, and $f\in L^p\II$ with $p>d$.
\end{assumption}
Note that Assumption \ref{ass:dis-reg} implies the regularity estimate
$u^\dag\equiv u(q^\dag) \in   H^2\II\cap W^{1,\infty}\II\cap H_0^1\II.$
This enables fully utilizing the approximation property of $V_h$ in the FEM analysis. We define two crucial quantities:
\begin{align}
\eta :=& h^2 + \gamma^{\frac12}\rho_0^{\frac{\kappa}{2}} + \gamma^{\frac d8}h^{4(\frac12-\frac d8)}\rho_0^{-s} + (\gamma^{\frac12 + \frac{d}{8}}h^{2-\frac d2}\rho_0^{2\zeta+1-s})^{\frac12} \nonumber \\ 
&+ \gamma^{\frac12}(\gamma^{-\frac12+\frac d8}h^{4(\frac12-\frac d8)}\rho_0^{-s})^{\frac{\kappa}{2(\kappa-2\zeta-1)}}, \label{eqn:eta}\\
\omega:=& (\gamma^{-1} h^4)^{\frac1\kappa} +  \rho_0 + (\gamma^{-\frac12+\frac d8}h^{4(\frac12-\frac d8)}\rho_0^{-s})^{\frac 2\kappa} + (\gamma^{-\frac12 + \frac{d}{8}}h^{2-\frac d2}\rho_0^{2\zeta-s})^{\frac 1\kappa}\nonumber\\
 &+ (\gamma^{-\frac12+\frac d8}h^{4(\frac12-\frac d8)}\rho_0^{-s})^{\frac{1}{\kappa-2\zeta-1}}.\label{eqn:omega}  
\end{align}
The following theorem gives the second main result of this work.
\begin{theorem}\label{thm:err-conduc}
Let Assumptions \ref{ass:quasi-points} and \ref{ass:dis-reg} hold. Let $q_h^*\in\mathcal{A}_h$ be a minimizer to problem \eqref{eqn:dis-optim-dis}-\eqref{eqn:dis-weak-conduc}. Let $\rho_0$ and $\gamma$ be given by \eqref{eqn:rho0-gamma}, and let $\eta$ and $\omega$ be given by \eqref{eqn:eta} and \eqref{eqn:omega}, respectively. Then with probability at least $1-2\tau$, with $\ell_\tau=\sqrt{\log\frac{2}{\tau}}$, there holds
\begin{equation*}
\int_\Omega\Big(\frac{q^\dag - q_h^*}{q^\dag}\Big)^2(fu^\dag + q^\dag|\nabla u^\dag|^2)\ {\rm d}x\leq   c(\eta + h^2 + n^{-\frac2d})^{\frac12}(\omega \ell_\tau)^{2\zeta+2}.
\end{equation*}
Moreover, if Condition \ref{Cond:positivity} holds, then
\begin{equation*}
    	\|q^{\dag}-q_h^*\|_{L^2\II}\leq  c\big[(\eta + h^2 + n^{-\frac2d})^{\frac12}(\omega \ell_\tau)^{2\zeta+2}\big]^{\frac{1}{2(1+\beta)}}.
    \end{equation*}
\end{theorem}
\begin{remark}
Theorem \ref{thm:err-conduc} provides guidelines on choosing the mesh size $h$ and the penalty parameter $\gamma$. By taking $h = O(\gamma^{\frac14})$ in Lemma \ref{lem:approx-sol-uppbound-qh}, we have $\eta\leq c\gamma^{\frac12}\rho_0^{2\zeta+1}$ and $\omega\leq c\rho^2_0$. Thus with $\gamma\geq n^{-\frac4d}$, there holds 
\begin{equation*}
\|q^{\dag}-q_h^*\|_{L^2\II}\leq  c\gamma^{\frac{1}{8(1+\beta)}}\rho_0^{\frac{10\zeta+9}{4(1+\beta)}}\ell_\tau^{\frac{1+\zeta}{1+\beta}}.  
\end{equation*}
The convergence rate agrees with the continuous case in Remark \ref{rem:con-rate}. Notably, the exponent of the  factor $\rho_0$ is larger than that in the continuous setting due to the repeated use of the $H^2\II$ regularity in the FEM analysis, cf. Lemmas \ref{lem:reg-u} and \ref{lem:point-err}. 
\end{remark}

\section{Proof of Theorem \ref{thm:con-err-conduc}}\label{sec:conv-cont}
To prove Theorem \ref{thm:con-err-conduc}, we first recall useful Sobolev embedding results. We use the notation $\hookrightarrow$ and $\overset{c}{\hookrightarrow}$ to denote continuous and compact embeddings, respectively.
\begin{lemma}\label{lem:Sob-embed}
The following embedding results hold for $d = 1, 2, 3$:
\begin{equation*}
W^{1,4}\II \hookrightarrow L^\infty\II \quad \mbox{and}\quad H^1\II\hookrightarrow L^4\II.
\end{equation*}
\end{lemma}

The following lemma gives the Sobolev regularity and Lipschitz continuity of the state $u(q)$ with respect to the parameter $q$.
\begin{lemma}\label{lem:reg-u}
Let $q, \widetilde{q}\in W^{1,4}\II\cap\mathcal{A}$, and $f\in L^2\II$. Then there exists $c$ independent of $q$ such that 
\begin{equation}\label{eqn:reg-u}
 \|u(q)\|_{H^2(\Omega)}\leq  c\|q\|^\zeta_{W^{1,4}\II}\|f\|_{L^2\II}.   
\end{equation}
Moreover, the following Lipschitz stability holds
\begin{equation}\label{eqn:cont-u}
    \|u(q)-u(\widetilde{q})\|_{H^2(\Omega)}\leq c\|\widetilde{q}-q\|_{W^{1,4}\II}\|\widetilde q\|_{W^{1,4}\II}^\zeta\|q\|_{W^{1,4}\II}^\zeta.
\end{equation}
\end{lemma}
\begin{proof}
We improve the solution regularity iteratively. Let $u\equiv u(q)$. Clearly we have
\begin{equation*}
   \left\{\begin{aligned}
    - \Delta u & = q^{-1}(f + \nabla q\!\cdot\! \nabla u), \ &\mbox{in}&\ \Omega, \\
    u&=0, \ &\mbox{on}&\ \partial\Omega.
	\end{aligned}
	\right.
\end{equation*}
By the full regularity pickup of the Dirichlet Laplacian $-\Delta$ in $W^{2,p}\II$ \cite[p. 105-111]{grisvard2011elliptic}, the box constraint of $\mathcal{A}$ and H\"{o}lder's inequality, there holds
\begin{equation}\label{reg-iter-ineq}
\|u\|_{W^{2,p}\II} \leq c(\|f\|_{L^p\II} + \|\nabla q\!\cdot\! \nabla u\|_{L^p\II})  \leq c(\|f\|_{L^p\II} + \|\nabla q\|_{L^4\II}\|\nabla u\|_{L^r\II}), 
\end{equation}
with the pair $(p,r)$ satisfying $\frac1p=\frac14+\frac1r$ to be determined. First,  integration by parts, the box constraint of the set $\mathcal{A}$ and the Cauchy-Schwarz inequality lead to 
\begin{equation*}
    c_0\|\nabla u\|_{L^2(\Omega)}^2 \leq \|q^\frac{1}{2}\nabla u\|_{L^2(\Omega)}^2 = (f,u) \leq \|f\|_{L^2(\Omega)}\|u\|_{L^2(\Omega)}.
\end{equation*}
This and Poincar\'{e} inequality imply the estimate $\|u\|_{H^1\II} \leq c\|f\|_{L^2\II}$. Let $(p,r) = (\frac{4}{3},2)$. In view of the estimate \eqref{reg-iter-ineq} and the embedding $W^{1,p}\II\hookrightarrow L^{\frac{p d}{d-p}}\II$ for $p<d$ \cite[Theorem 4.12, Case C]{Adams2003Sobolev}, there holds for $d=2$
\begin{equation}\label{iter-ineq-1st-step in 2d}
\|u\|_{W^{1,4}\II}\leq c\|u\|_{W^{2,\frac{4}{3}}\II} \leq c(1+\|\nabla q\|_{L^4\II})\|f\|_{L^2\II}.    
\end{equation}
Then choosing $(p,r) = (2,4)$ and using the estimates \eqref{reg-iter-ineq} and \eqref{iter-ineq-1st-step in 2d} yield
\begin{equation*}
\|u\|_{H^2\II} \leq c(\|f\|_{L^2\II} + \|\nabla q\|_{L^4\II}\|\nabla u\|_{L^4\II})\leq c\bigg(\sum_{k=0}^\zeta\|q\|_{W^{1,4}\II}^k\bigg)\|f\|_{L^2\II}.
\end{equation*}
The case $d=3$ follows analogously. In sum, we obtain 
\begin{equation*}
    (p,r) = \big\{(\tfrac{4}{3},2),~ (2,4) \big\}\ \mbox{for}\ d=2 \quad \mbox{and}\quad (p,r) = \big\{(\tfrac{4}{3},2),~ (\tfrac{3}{2},\tfrac{12}{5}), (\tfrac{12}{7},3), (2,4) \big\}\ \mbox{for}\ d=3.
\end{equation*}
Thus by the definition of $\zeta$ and the box constraint of $\mathcal{A}$, we obtain
\begin{equation*}
\|u\|_{H^2(\Omega)}\leq c\Big(\sum_{k=0}^{\zeta}\|q\|^k_{W^{1,4}\II}\Big)\|f\|_{L^2\II} \leq 
c\|q\|^\zeta_{W^{1,4}\II}\|f\|_{L^2\II}.
\end{equation*}
Next, let $w:=u-\widetilde{u}\equiv u(q) - u(\widetilde{q})$. Then the function $w$ satisfies
\begin{equation*}
   \left\{\begin{aligned}
    -\nabla\!\cdot\!(q\nabla w) & = -\nabla\!\cdot\![(\widetilde{q}-q)\nabla \widetilde{u}], \ &\mbox{in}&\ \Omega, \\
    w&=0, \ &\mbox{on}&\ \partial\Omega.
\end{aligned}
\right.
\end{equation*}
By H\"{o}lder's inequality, Lemma \ref{lem:Sob-embed}, and the estimate \eqref{eqn:reg-u}, we arrive at 
\begin{align*}
\|\nabla\!\cdot\![(\widetilde{q}-q)\nabla \widetilde{u}]\|_{L^2\II} &\leq \|\widetilde{q}-q\|_{L^\infty\II}\|\Delta \widetilde{u}\|_{L^2\II}  + \|\nabla(\widetilde{q}-q)\|_{L^4\II}\|\nabla \widetilde{u}\|_{L^4\II} \\ & \leq c\|\widetilde{q}-q\|_{W^{1,4}\II}\|\widetilde{u}\|_{H^2\II} \leq c\|\widetilde{q}-q\|_{W^{1,4}\II}\|\widetilde q\|^\zeta_{W^{1,4}\II}.
\end{align*}
Then we derive
\begin{align*}
    \|w\|_{H^2\II} & \leq c\|q\|^\zeta_{W^{1,4}\II}\|\nabla\!\cdot\![(\widetilde{q}-q)\nabla \widetilde{u}]\|_{L^2\II} \leq c\|\widetilde{q}-q\|_{W^{1,4}\II}\|q\|^\zeta_{W^{1,4}\II}\|\widetilde q\|^\zeta_{W^{1,4}\II}.
\end{align*}
This completes the proof of the lemma.
\end{proof}

The next result gives the existence of a minimizer $q^*$ to problem \eqref{eqn:conti-optim-prob-conduc}-\eqref{eqn:conti-weak-conduc}.
\begin{proposition}\label{prop:con-wellposedness}
Under Assumption \ref{ass:con-reg}, there exists at least one minimizer $q^*\in\mathcal{A}\cap W^{1,4}(\Omega)$ to problem \eqref{eqn:conti-optim-prob-conduc}-\eqref{eqn:conti-weak-conduc} $\mathbb{P}$ almost surely.
\end{proposition}
\begin{proof}
Fix any realization of the random data $\bm{m}$. By Assumption \ref{ass:con-reg} and Lemma \ref{lem:reg-u}, we have $u(q)\in H^2\II\hookrightarrow C(\overline{\Omega})$. Thus, the pointwise evaluation in $J_\gamma$ in \eqref{eqn:conti-optim-prob-conduc} makes sense. By the non-negativity of $J_\gamma$, there exists a minimizing sequence $(q^k)_{k=1}^\infty\subset\mathcal{A}$ such that
\begin{equation*}
    \lim_{k\to\infty}J_\gamma(q^k) = \inf_{q\in \mathcal{A}}J_\gamma(q).
\end{equation*}
The uniform boundedness of $(q^k)_{k=1}^\infty$ in $W^{1,4}\II$ and Lemma \ref{lem:reg-u} imply the uniform boundedness of $(u(q^k) )_{k=1}^\infty $ in $H_0^1(\Omega)\cap H^2(\Omega)$. One can extract two subsequences from $(q^k)_{k=1}^\infty$ and $(u(q^k))_{k=1}^\infty$, still denoted by $(q^k)_{k=1}^\infty$ and $(u(q^k))_{k=1}^\infty$, and some $q^*\in\mathcal{A}$ and $u^*\in H^2\II$ such that
\begin{equation*}
q^k \rightharpoonup q^* \mbox{ in } W^{1,4}\II\quad\mbox{and} \quad u(q^k)\rightharpoonup u^*\mbox{ in } H^2\II.
\end{equation*}
We claim the identity $u^* \equiv u(q^*)$. Indeed, by the weak formulation of $u(q^k)$, there holds
\begin{equation}\label{min-sequence-eqn}
    (q^k\nabla u(q^k), \nabla \varphi) = (f,\varphi), \quad \forall\varphi\in H_0^1\II.
\end{equation}
The box constraint of $\mathcal{A}$, H\"{o}lder's inequality and the compact Sobolev embedding $H^2\II\hookrightarrow H^1\II$ give
\begin{align*}
    &\lim_{k\to\infty}|(q^k\nabla (u(q^k) - u^*), \nabla \varphi)| \leq \lim_{k\to\infty}\|q^k\|_{L^\infty\II}\|\nabla (u(q^k) - u^*)\|_{L^2\II} \|\nabla\varphi\|_{L^2\II}=0.
\end{align*}
Similarly, the embeddings $ W^{1,4}\II\overset{c}{\hookrightarrow} L^4\II$ and $H^2\II\hookrightarrow W^{1,4}\II$ for $d=2, 3$ lead to
\begin{align*}
\lim_{k\to\infty}|((q^k-q^*)\nabla u^*, \nabla \varphi)| & \leq \lim_{k\to\infty}\|q^k-q^*\|_{L^4\II} \|\nabla u^*\|_{L^4\II}\|\nabla\varphi\|_{L^2\II} \\& \leq 
\lim_{k\to\infty}\|q^k-q^*\|_{L^4\II} \| u^*\|_{H^2\II}\|\nabla\varphi\|_{L^2\II}
= 0.    
\end{align*}
These two estimates and \eqref{min-sequence-eqn} yield the identity
\begin{equation*}
    (q^*\nabla u^*,\nabla \varphi) = (f,\varphi), \quad \forall\varphi\in H_0^1\II.
\end{equation*}
The definition of $u(q^*)$ implies  $u^* \equiv u(q^*)$. By the weak lower semicontinuity of the $W^{1,4}\II$ norm and the compact embedding $H^2(\Omega)\overset{c}{\hookrightarrow} C(\overline{\Omega})$, $q^*$ is a minimizer to problem \eqref{eqn:conti-optim-prob-conduc}-\eqref{eqn:conti-weak-conduc}.  Moreover, the minimizer $q^*$ depends continuously on the data perturbation (see, e.g., \cite[Theorem 4.2, p. 109]{ItoJin:2015} for the argument). 
Since the existence of a minimizer holds for any realization of the random data $\bm{m}$, the desired statement holds $\mathbb{P}$ almost surely.
\end{proof}

We use extensively the following 
lemma due to van de Geer on
stochastic convergence \cite[Lemma 8.4 and (10.6)]{geer2000empirical}. Using the terminology of the stochastic
convergence order, we denote a random variable $X = \mathcal{O}_p(z)$ if $X$ is a sub-Gaussian random
variable with zero expectation and parameter $z$.
\begin{lemma}
\label{Geer-entropy}
Let $G$ be a function space with $\log N(\varepsilon, B_G,\|\cdot\|_{n}) \leqslant  A \varepsilon^{-r}$, with $N(\varepsilon, B_G,\|\cdot\|_{n})$ being the local covering number of the unit ball $B_G$. Then there holds
\begin{equation*}
    \sup_{g \in G} \frac{\left|(\bm{\xi}, g)_{n}\right|}{\|g\|_{n}^{1-\frac r2}\|g\|_{G}^{\frac r2}}=\mathcal{O}_{p}\left( \sqrt{A} \sigma n^{-\frac12}\right).
\end{equation*}
\end{lemma}

The next lemma provides $\psi_2$-Orlicz bounds on $\|u^\dagger-u(q^*)\|_n$ and $\| q^*\|_{W^{1,4}(\Omega)}$. See the appendix for the definition of the Orlicz norm $\|\cdot\|_{\psi_2}$.
\begin{lemma}\label{lem:apprx-u-bound-q}
Let Assumptions \ref{ass:quasi-points} and \ref{ass:con-reg} hold, and let $q^*$ be a minimizer to problem \eqref{eqn:conti-optim-prob-conduc}-\eqref{eqn:conti-weak-conduc}. Then with $\rho_0$ and $\gamma$ given by \eqref{eqn:rho0-gamma}, the following $\psi_2$-Orlicz bounds hold
\begin{equation*}
\Big\|\|u(q^*)-u^\dagger\|_n\Big\|_{\psi_2} \leq c\gamma^\frac12 \rho_0^{\frac\kappa2}\quad\mbox{and}\quad\Big\|\| q^*\|_{W^{1,4}(\Omega)}\Big\|_{\psi_2}\leq c\rho_0.
\end{equation*}
\end{lemma}
\begin{proof}
By the minimizing property of $q^* \in \mathcal{A}$ to the functional $J_\gamma(q) $ over $\mathcal{A}$, it follows that
	\begin{equation*}
		\|u(q^*)-\bm m\|_{n}^2 + \gamma \|q^*\|_{W^{1,4}\II}^\kappa\leq \|u^\dagger-\bm m\|_{n}^2 + \gamma\| q^\dagger\|_{W^{1,4}\II}^\kappa.
	\end{equation*}
	Then direct computation leads to the elementary inequality
	\begin{align}\label{eqn:min-prop}
		\|u(q^*)-u^\dagger\|_n^2 + \gamma \|q^*\|_{W^{1,4}\II}^\kappa & \leq 2\big(\bm\xi,u(q^*) - u^\dagger\big)_n +  \gamma\|q^\dag\|_{W^{1,4}\II}^\kappa=:F_1 + F_2.
	\end{align}
Then one of the following two events $F_1$ and $F_2$ must happen
\begin{equation*}
    \|u(q^*)-u^\dagger\|_n^2 + \gamma \|q^*\|_{W^{1,4}\II}^\kappa \leq cF_i, \quad i=1,2,
\end{equation*}
and the events are also denoted by $F_i$.
Since $F_2$ is deterministic, when $F_2$ happens, there holds
\begin{equation*}
    \Big\|\|u(q^*)-u^\dagger\|_n\Big\|_{\psi_2} \leq c\gamma^{\frac12}\rho_0^\frac{\kappa}{2}\quad\mbox{and}\quad \Big\|\|q^*\|_{W^{1,4}\II}\Big\|_{\psi_2} \leq c\rho_0.
\end{equation*}
It remains to estimate the event $F_1$. First, by Lemma \ref{lem:reg-u}, we get 
\begin{equation*}
  \|u(q^*) - u^\dag\|_{H^2\II} \leq c \|q^\dag\|_{W^{1,4}\II}^\zeta \|q^*\|_{W^{1,4}\II}^\zeta \|q^\dag - q^*\|_{W^{1,4}\II}.
\end{equation*}
Then Lemmas \ref{lem:cover-en-bound-Sobolev} and \ref{Geer-entropy} imply 
\begin{equation}\label{con-F1-estimate}
    \begin{split}
    F_1 & \leq \mathcal{O}_p\Big(\sigma n^{-\frac12}\|u(q^*) - u^\dag\|_n^{1-\frac d4}\|u(q^*) - u^\dag\|_{H^2\II}^{\frac d4}\Big) \\
    & \leq \mathcal{O}_p\Big(\sigma n^{-\frac12}\|u(q^*) - u^\dag\|_n^{1-\frac d4}\big[\|q^\dag\|_{W^{1,4}\II}^\zeta \|q^*\|_{W^{1,4}\II}^\zeta \|q^\dag - q^*\|_{W^{1,4}\II}\big]^{\frac d4}\Big).    
    \end{split}
\end{equation}
Now consider the following two sub-cases separately: (i) $\|q^*\|_{W^{1,4}\II}\leq \|q^\dag\|_{W^{1,4}\II}$ and (ii) $\|q^*\|_{W^{1,4}\II}\geq \|q^\dag\|_{W^{1,4}\II}$. If the event (i) occurs, the estimate $\Big\|\|q^*\|_{W^{1,4}\II}\Big\|_{\psi_2} \leq \rho_0$ follows directly. Further, we obtain
\begin{equation}\label{con-event-i}
\|u(q^*)-u^\dagger\|_n^2 \leq \mathcal{O}_p\Big(\sigma n^{-\frac12}\Big)\|u(q^*) - u^\dag\|_n^{1-\frac d4}\rho_0^{\frac{(2\zeta+1)d}{4}}.  
\end{equation}
The choice $\kappa = 4\zeta$ and $\gamma^{\frac12 + \frac d8}= O\Big(\sigma n^{-\frac12}\rho_0^{\frac{-2\kappa+d}{4}}\Big)$ gives
\begin{equation}\label{con-event-i-state-approx}
\Big\|\|u(q^*)-u^\dagger\|_n\Big\|_{\psi_2} \leq c\big(\sigma n^{-\frac12}\big)^{\frac{4}{4+d}}\rho_0^{\frac{(2\zeta +1)d}{4+d}} \leq c\gamma^{\frac12}\rho_0^{\frac\kappa 2}.    
\end{equation}
If the event (ii) occurs, from \eqref{con-F1-estimate}, we have
\begin{equation}\label{dis-event-ii}
\|u(q^*)-u^\dagger\|_n^2 + \gamma \|q^*\|_{W^{1,4}\II}^\kappa \leq \mathcal{O}_p\Big(\sigma n^{-\frac12}\Big)\|u(q^*) - u^\dag\|_n^{1-\frac d4}\|q^*\|_{W^{1,4}\II}^{\frac{(2\zeta +1)d}{4}}.    
\end{equation}
This directly implies that the following two inequalities
\begin{equation}\label{con-event-ii-alternative-ineq}
\left\{\begin{aligned}
\|u(q^*) - u^\dag\|_n^{1+\frac d4} & \leq \mathcal{O}_p\Big(\sigma n^{-\frac12}\Big)\|q^*\|_{W^{1,4}\II}^{\frac{(2\zeta +1)d}{4}}, \\
     \gamma \|q^*\|_{W^{1,4}\II}^{\kappa -  \frac{(2\zeta +1)d}{4}} &\leq \mathcal{O}_p\Big(\sigma n^{-\frac12}\Big)\|u(q^*) - u^\dag\|_n^{1-\frac d4}.    
\end{aligned}\right.    
\end{equation}
Then by solving the system with the choice $\kappa = 4\zeta$ and $\gamma^{\frac12 + \frac d8}= O\Big(\sigma n^{-\frac12}\rho_0^{\frac{-2\kappa+d}{4}}\Big)$, we deduce
\begin{equation*}
\Big\|\|u(q^*)-u^\dagger\|_n\Big\|_{\psi_2} \leq c\gamma^{\frac12}\rho_0^{\frac\kappa 2}\quad\mbox{and}\quad\Big\|\|q^*\|_{W^{1,4}\II}\Big\|_{\psi_2} \leq c\rho_0.    
\end{equation*}
Combining the preceding estimates gives the desired estimates and completes the proof of the lemma.
\end{proof}

The following result connects the standard $L^2\II$ norm with the discrete semi-norm $\|\cdot\|_n$ for every $v\in H^2(\Omega)$ \cite[Theorems 3.3 and 3.4]{utreras1988convergence}. 
\begin{lemma}\label{lem:connect-seminorm-stdnorm}
	Under Assumption \ref{ass:quasi-points}, there exists $c=c(\Omega,B)$ such that the following estimates hold
	\begin{equation*}
		\|v\|_{L^2\II}\le c\big(\|v\|_n+ n^{-\frac{2}{d}}|v|_{H^2\II}\big)\quad \mbox{and}\quad \|v\|_n\le c\big(\|v\|_{L^2\II}+ n^{-\frac{2}{d}}|v|_{H^2\II}\big), \quad\forall v\in H^2\II.
	\end{equation*}
\end{lemma}

Now we can give the proof of Theorem \ref{thm:con-err-conduc}.

\begin{proof}
By Lemma \ref{lem:apprx-u-bound-q} and the estimate \eqref{ineq:proba-distrib-estimate}, the following two estimates hold simultaneously 
\begin{equation*}
\|u(q^*)-u^\dagger\|_n \leq c\gamma^{\frac12}\rho_0^{\frac\kappa2}\ell_\tau \quad \mbox{and}\quad \| q^*\|_{W^{1,4}\II} \leq c\rho_0\ell_\tau,
\end{equation*}
with probability at least $1-2\tau$.
By the weak formulations of $u^\dag\equiv u(q^\dag)$ and $u(q^*)$, we have 
\begin{equation*}
    ((q^\dag - q^*)\nabla u^\dag,\nabla\varphi) = (q^*\nabla (u(q^*)-u^\dag),\nabla\varphi),\quad \forall \varphi\in H_0^1\II.
\end{equation*}
Let $\varphi = \frac{q^\dag - q^*}{q^\dag} u^\dag$. Then the box constraint of $\mathcal{A}$, H\"{o}lder's inequality and Lemma \ref{lem:Sob-embed} lead to
\begin{equation}\label{estimate-grad-varphi}
    \begin{split}
     \|\nabla\varphi\|_{L^2\II} & \leq \bigg\| \frac{\nabla (q^\dag - q^*)}{q^\dag} u^\dag - \frac{ (q^\dag - q^*)\nabla q^\dag}{(q^\dag)^2} u^\dag + \frac{q^\dag - q^*}{q^\dag}\nabla u^\dag\bigg\|_{L^2\II} \\
    & \leq c\big[\big( \|\nabla q^\dag\|_{L^4\II} +  \|\nabla q^*\|_{L^4\II}\big)\|u^\dag\|_{L^4\II} + \|\nabla u^\dag\|_{L^2\II} \big] \\& \leq c\big( \|q^\dag\|_{W^{1,4}} +  \|q^*\|_{W^{1,4}\II}\big)\|u^\dag\|_{H^1\II}  \leq c\rho_0\ell_\tau.   
    \end{split}
\end{equation}
Thus Poincar\'{e} inequality directly gives $\varphi\in H_0^1\II$. From Lemma \ref{lem:reg-u}, we get
\begin{equation*}
\|u(q^*) - u^\dag\|_{H^2\II} \leq \|q^*-q^\dag\|_{W^{1,4}\II}\| q^*\|_{W^{1,4}\II}^\zeta\|q^\dag\|_{W^{1,4}\II}^\zeta \leq c( \rho_0\ell_\tau)^{2\zeta+1}.   
\end{equation*}
By the Gagliardo–Nirenberg interpolation inequality \cite{Brezis2018}, Lemma \ref{lem:connect-seminorm-stdnorm} and the elementary inequality $\frac\kappa2\leq 2\zeta+1$, we obtain
\begin{align*}
& \|\nabla (u(q^*) - u^\dag)\|_{L^2\II}  \leq c\|u(q^*) - u^\dag\|_{L^2\II}^{\frac12}\|u(q^*) - u^\dag\|_{H^2\II}^{\frac12} \\
 \leq & c\big[\|u(q^*) - u^\dag\|_n + n^{-\frac2d}\|u(q^*) - u^\dag\|_{H^2\II}\big]^{\frac12}\|u(q^*) - u^\dag\|_{H^2\II}^{\frac12} \\
\leq & c\big[\gamma^{\frac12}\rho_0^{\frac\kappa2}\ell_\tau + n^{-\frac2d}(\rho_0\ell_\tau)^{2\zeta+1}\big]^{\frac12}(\rho_0\ell_\tau)^{\frac{2\zeta+1}{2}} \leq c\big(\gamma^{\frac12} + n^{-\frac2d}\big)^{\frac12}(\rho_0\ell_\tau)^{2\zeta+1}.
\end{align*}
Then H\"{o}lder's inequality and the box constraint of the set $\mathcal{A}$ give
\begin{align*}
    |(q^*\nabla (u(q^*)-u^\dag),\nabla\varphi)| \leq c\|\nabla (u(q^*) - u^\dag)\|_{L^2\II} \|\nabla\varphi\|_{L^2\II} \leq c\big(\gamma^{\frac12} + n^{-\frac2d}\big)^{\frac12}(\rho_0\ell_\tau)^{2\zeta+2}.
\end{align*}
Meanwhile, the argument in \cite{Bonito:2017} yields the identity
\begin{equation*}
((q^\dag - q^*)\nabla u^\dag,\nabla\varphi) = \frac12\int_{\Omega}\Big(\frac{q^\dag - q^*}{q^\dag}\Big)^2(fu^\dag + q^\dag|\nabla u^\dag|^2)\ {\rm d}x.    
\end{equation*}
Thus the first assertion follows. 
Next we divide the domain $\Omega$ into two disjoint sets $\Omega_r=\{x\in\Omega:{\rm dist}(x,\partial\Omega)\geq r\}$ and $\Omega_r^c=\Omega\setminus\Omega_r$, with the constant $r>0$ to be chosen. On the subdomain $\Omega_r$, the first assertion yields
\begin{align*}
    \|q^\dag - q^*\|_{L^2(\Omega_r)}^2 \leq cr^{-\beta}\int_{\Omega_r}\Big(\frac{q^\dag - q^*}{q^\dag}\Big)^2(fu^\dag + q^\dag|\nabla u^\dag|^2)\ {\rm d}x  \leq cr^{-\beta}\big(\gamma^{\frac12} + n^{-\frac2d}\big)^{\frac12}(\rho_0\ell_\tau)^{2\zeta+2}.
\end{align*}
By the box constraint of $\mathcal{A}$, we deduce 
\begin{equation*}
\|q^\dag - q^*\|_{L^2(\Omega_r^c)}^2 \leq c|\Omega_r^c| \leq cr.    
\end{equation*}
Then balancing $r$ with $ r^{-\beta}\big(\gamma^{\frac12} + n^{-\frac2d}\big)^{\frac12}(\rho_0\ell_\tau)^{2\zeta+2}$ yields the desired estimate. 
\end{proof}

\section{Proof of Theorem \ref{thm:err-conduc}}
\label{sec:conv-disc}
The proof of Theorem \ref{thm:err-conduc} requires several estimates in the FEM analysis.
Let $\Pi_h$ be the Lagrange nodal interpolation operator on the space $V_h$. Then the following estimates hold for $s\in[1,2]$ and $p\in[1,\infty]$ with $sp>d$ and $k\leq s$ \cite[Corollary 4.4.24]{BrennerScott:book2008}:
\begin{align}\label{ineq:Pi_h-approx}
\|v-\Pi_hv\|_{L^\infty(K)} &\leq ch^{s-\frac dp} \|v\|_{W^{s,p}(K)},\quad \forall v\in W^{s,p}(K),  \\
\|v-\Pi_hv\|_{W^{k,p}\II} &\leq ch^{s-k} \|v\|_{W^{s,p}\II}, \quad \forall v\in W^{s,p}\II.
\end{align}
Moreover, we define the $L^2\II$ projection operator $P_h:L^2\II\to V_h$ by
\begin{equation*}
	(P_hv,v_h)=(v,v_h), \quad \forall v\in L^2\II,\ v_h\in V_h.
\end{equation*}
Then for $1\leq p\leq \infty$ and $0\leq k\leq s\leq 2$ \cite[pp. 85--86]{Thome2006GalerkinFE}:
\begin{equation}\label{ineq:P_h-approx}
	\big\|v-P_hv\big\|_{W^{k,p}\II}\leq  ch^{s-k}\big\|v\big\|_{W^{s,p}\II}, \quad \forall v\in W^{s,p}\II.
\end{equation}
The estimate \eqref{ineq:P_h-approx} of $P_h$, the inverse estimate in $V_h$ \cite[Chapter 4]{BrennerScott:book2008} and \eqref{ineq:Pi_h-approx} of $\Pi_h$ imply that for any $v\in W^{1,4}\II$,
\begin{align}
&\|\Pi_h v\|_{W^{1,4}\II} \leq \|P_h v\|_{W^{1,4}\II}  + \|P_hv - \Pi_h v\|_{W^{1,4}\II} \nonumber\\
\leq& c\| v\|_{W^{1,4}\II} + ch^{-1}\|P_hv - \Pi_h v\|_{L^4\II} \leq c\|v\|_{W^{1,4}\II}. \label{W14-stab-Pih}       
\end{align}
Additionally, we employ an $H^2(\Omega)$ conforming FEM space $W_h$ (e.g., Hermite elements and Argyris elements for $d=1,2$ \cite{BrennerScott:book2008} and Zhang elements for $d=3$ \cite{zhang2009family}). We define an interpolation operator $\widetilde{\Pi}_h: H^2\II\mapsto W_h \subset H^2(\Omega)$ such that for any $0\leq k\leq s\leq 2$, there holds
\begin{equation}\label{eqn:w-Pi-h}
\|v-\widetilde{\Pi}_h v\|_{H^k\II} \leq ch^{s-k}\|v\|_{H^s\II},\quad \forall v\in H^s(\Omega).    
\end{equation}

Next we give several technical estimates in the $\|\cdot\|_n$ norm. Note that when $n$ is large, the semi-norm $\|\cdot\|_n$ is actually a norm on the FEM space $V_h$ and is equivalent to the standard $\|\cdot\|_{L^2\II}$ norm \cite[Lemma 5.3]{jin2024stochastic}.
\begin{lemma}\label{lem:dis-equiv-n-L2}
Let Assumption \ref{ass:quasi-points} hold. Then there exist $c,c'>0$, independent of $n$ and $h$, such that
	\begin{equation*}
		c'\|v_h\|_{L^2\II}\leq\|v_h\|_{n}\leq c\|v_h\|_{L^2\II},\quad\forall v_h\in V_h.
	\end{equation*}
\end{lemma}
The next result gives error bounds on the FEM approximations $u_h(q)$ and $u_h(\Pi_hq^\dag)$ in the $\|\cdot\|_n$ norm.
\begin{lemma}\label{lem:point-err}
Let Assumption \ref{ass:quasi-points} hold. Then for any $q\in W^{1,4}\II\cap\mathcal{A}$ and $f\in L^2(\Omega)$, there holds
\begin{equation*}
\|u(q) - u_h(q)\|_n  \leq ch^2\|q\|^{2\zeta}_{W^{1,4}\II}.  
\end{equation*}	
If Assumption \ref{ass:dis-reg} holds, then there exists $c>0$, depending on $\|q^\dag\|_{H^2(\Omega)}$ and $\|q^\dag\|_{W^{1,\infty}\II}$, such that
\begin{align*}
 \|u^\dag-u_h(\Pi_hq^{\dag})\|_{n} \leq ch^2.
\end{align*}
\end{lemma}
\begin{proof}
By the argument of \cite[Lemma 5.4]{jin2024stochastic} and Lemma \ref{lem:reg-u}, we deduce
\begin{equation*}
    \|u(q) - \Pi_hu(q)\|_n \leq ch^2\|u(q)\|_{H^2(\Omega)} \leq  ch^2\|q\|^{\zeta}_{W^{1,4}\II}.
\end{equation*}
Let $e:=u(q) - u_h(q)$. Then C\'{e}a's lemma \cite[Chapter 5]{BrennerScott:book2008} and Lemma \ref{lem:reg-u} yield 
\begin{equation*}
\|e\|_{H^1\II}\leq ch\|u(q)\|_{H^2\II} \leq ch\|q\|^{\zeta}_{W^{1,4}\II}. 
\end{equation*}
To bound $e$ in the $L^2\II$ norm, we employ Nitsche's trick \cite[Chapter 5]{BrennerScott:book2008}. Let $z\equiv z(q)\in H^2\II\cap H_0^1\II$ and $z_h\equiv z(q)\in X_h$ solve respectively
\begin{equation*}
\begin{split}
&(q\nabla z,\nabla \varphi) = (e,\varphi), \quad \forall\varphi\in H_0^1(\Omega), \\
&(q\nabla z_h,\nabla \varphi_h) = (e,\varphi_h), \quad \forall\varphi_h\in X_h.
\end{split}
\end{equation*}
Let $\varphi= u(q)$ and $\varphi_h = u_h(q)$. Thus Galerkin orthogonality and Lemma \ref{lem:reg-u} imply
\begin{align*}
    \|e\|_{L^2\II}^2 \leq ch^2\|z(q)\|_{H^2\II}\|u(q)\|_{H^2\II} \leq ch^2 \|q\|^{2\zeta}_{W^{1,4}\II}\|e\|_{L^2\II}. 
\end{align*}
Hence $\|e\|_{L^2\II}\leq ch^2 \|q\|^{2\zeta}_{W^{1,4}\II}$. By Lemma \ref{lem:dis-equiv-n-L2} and the triangle inequality, we obtain
\begin{align*}
    \|\Pi_hu(q) - u_h(q)\|_n & \leq c\|\Pi_hu(q) - u_h(q)\|_{L^2\II}
    \leq c\big(\|\Pi_hu(q) - u(q)\|_{L^2\II}+ \|e\|_{L^2\II}\big)   \\
    & \leq ch^2\|q\|^{\zeta}_{W^{1,4}\II} + ch^2\|q\|^{2\zeta}_{W^{1,4}\II} \leq ch^2\|q\|^{2\zeta}_{W^{1,4}\II}.
\end{align*}
Now the first estimate follows from the triangle inequality. Next, under Assumption \ref{ass:dis-reg}, we have \cite[Lemma A.1]{JinZhou:2021SINUM}
\begin{equation*}
\|u^\dag-u_h(\Pi_hq^{\dag})\|_{L^2\II} \leq ch^2.    
\end{equation*}
Finally, the splitting $u^\dag-u_h(\Pi_hq^{\dag}) = (u^\dag - \Pi_hu^\dag) + (\Pi_hu^\dag-u_h(\Pi_hq^{\dag}))$ and Lemma \ref{lem:dis-equiv-n-L2}
give the second assertion.
\end{proof}

The next lemma gives a discrete analogue of Lemma \ref{lem:apprx-u-bound-q}. 
\begin{lemma}\label{lem:approx-sol-uppbound-qh}
Let Assumptions \ref{ass:quasi-points} and \ref{ass:dis-reg} hold, and let $q_h^*\in\mathcal{A}_h$ be a minimizer to problem \eqref{eqn:dis-optim-dis}-\eqref{eqn:dis-weak-conduc}. 
Then with $\rho_0$ and $\gamma$ given by \eqref{eqn:rho0-gamma}, the following $\psi_2$-Orlicz bounds hold
\begin{align*}
\Big\|\|u_h(q_h^*)-u_h(\Pi_hq^\dagger)\|_n \Big\|_{\psi_2}\leq c\eta \quad\mbox{and}\quad \Big\|\|q_h^*\|_{W^{1,4}\II}\Big\|_{\psi_2}  \leq c \omega,
\end{align*}
with $\eta$ and $\omega$ defined in \eqref{eqn:eta} and \eqref{eqn:omega}, respectively.
Moreover, with the choice $h=O(\gamma^{\frac14})$, there hold
\begin{align*}
\Big\|\|u_h(q_h^*)-u_h(\Pi_hq^\dagger)\|_n \Big\|_{\psi_2}\leq c\gamma^{\frac12}\rho_0^{{\frac\kappa2}+1}\quad\mbox{and}\quad \Big\|\|q_h^*\|_{W^{1,4}\II}\Big\|_{\psi_2}  \leq c\rho_0^2.
\end{align*}
\end{lemma}

\begin{proof}
 Since $q_h^*\in\mathcal{A}_h$ is a minimizer and $\Pi_hq^\dagger\in\mathcal{A}_h$, there holds
\begin{equation*}
	\|u_h(q_h^*)-\bm m\|_{n}^2 + \gamma \|q_h^*\|_{W^{1,4}\II}^\kappa\leq \|u_h(\Pi_hq^\dagger)-\bm m\|_{n}^2 + \gamma\|\Pi_hq^\dagger\|_{W^{1,4}\II}^\kappa.
\end{equation*}
By the definition of the data $\bm{m}$ and the Cauchy-Schwarz inequality, we obtain 
\begin{align*}
\|u_h(q_h^*)-u_h(\Pi_hq^\dagger)\|_n^2 + \gamma \| q_h^*\|_{W^{1,4}\II}^\kappa & \leq \sum_{i=1}^5F_i,
\end{align*}
with the five terms respectively given by 
\begin{align*}
& F_1 := 2\big(\bm\xi,u_h(q_h^*) - \widetilde{\Pi}_hu(q_h^*)\big)_n, \\
& F_2 := 2\big(\bm\xi,\widetilde{\Pi}_hu(q_h^*) - \widetilde{\Pi}_hu(\Pi_hq^\dagger)\big)_n, \\
& F_3 := 2\big(\bm\xi, \widetilde{\Pi}_hu(\Pi_hq^\dagger) - u_h(\Pi_hq^\dagger)\big)_n, \\
& F_4 := 2\|u^\dagger-u_h(\Pi_hq^{\dagger})\|_n\|u_h(q_h^*) - u_h(\Pi_hq^\dagger)\|_n, \\
& F_5 := \gamma \|\Pi_hq^\dag\|_{W^{1,4}\II}^\kappa.
\end{align*}
Then one of the following five events must happen: 
\begin{equation*}
\|u_h(q_h^*)-u_h(\Pi_hq^\dagger)\|_n^2 + \gamma \| q_h^*\|_{W^{1,4}\II}^\kappa \leq c F_i, \quad i=1,\ldots,5,    
\end{equation*}
and the events are also denoted by $F_i$. If the event $F_4$ happens, by Lemma \ref{lem:point-err}, we have
\begin{equation*}
\|u_h(q_h^*)-u_h(\Pi_hq^\dagger)\|_n\leq ch^2 \quad \mbox{and}\quad  \| q_h^*\|_{W^{1,4}\II} \leq c(\gamma^{-1}h^4)^\frac{1}{\kappa}.    
\end{equation*}
In the event of $F_5$, by the definition of $\rho_0$ and the $W^{1,4}(\Omega)$ stability \eqref{W14-stab-Pih} of the operator $\Pi_h$, we derive
\begin{equation*}
\|u_h(q_h^*)-u_h(\Pi_hq^\dagger)\|_n^2 + \gamma \| q_h^*\|_{W^{1,4}\II}^\kappa \leq c\gamma\rho_0^\kappa.    
\end{equation*}
By combining bounds for the events $F_4$ and $F_5$, we deduce
\begin{equation*}
\Big\|\|u_h(q_h^*)-u_h(\Pi_hq^\dagger)\|_n \Big\|_{\psi_2} \leq c(h^2 + \gamma^{\frac12}\rho_0^{\frac{\kappa}{2}})\quad\mbox{and}\quad \Big\|\| q_h^*\|_{W^{1,4}\II}\Big\|_{\psi_2} \leq c \big((\gamma^{-1} h^4)^{\frac1\kappa} +  \rho_0\big).    
\end{equation*}
Thus it suffices to bound the remaining three terms $F_1$, $F_2$ and $F_3$ separately. The lengthy proof is divided into three steps below, each dealing with one term.
\\
\textbf{Step 1. Bound $F_1$.} We define an auxiliary function $F_1^d(q_h)$ by 
\begin{equation*}
F_1^d(q_h):= \frac{u_h(q_h) - \widetilde{\Pi}_hu(q_h)}{\|u_h(q_h)-u_h(\Pi_hq^\dagger)\|_n + \gamma^{\frac12} \| q_h\|_{W^{1,4}\II}^{\frac\kappa 2}}\in X_h + W_h.
\end{equation*}
First, we claim that $\{\big(\bm{\xi}, F_1^d(q_h) \big)_n, q_h\in\mathcal{A}_h \}$ is a sub-Gaussian random process equipped with the semi-metric
\begin{equation}\label{discrete-semidis}
    \widehat{\rbd}(q_h,\widetilde{q}_h):= \sigma n^{-\frac12}\| F_1^d(q_h) -  F_1^d(\widetilde{q}_h)\|_n, \quad \forall q_h,\widetilde{q}_h \in \mathcal{A}_h.
\end{equation}
Since the random variables $(\xi_j)_{j=1}^n$ are  i.i.d. sub-Gaussian with parameter $\sigma$ and zero expectation, the following estimate holds
\begin{align*} 
&\mathbb{E}\Big[\exp\Big(\lambda\big(\bm{\xi}, F_1^d(q_h) - F_1^d(\widetilde{q}_h) \big)_n  \Big)\Big] = \mathbb{E}\Big[\exp\Big(\lambda n^{-1}\sum_{j=1}^n\xi_j\big(F_1^d(q_h) - F_1^d(\widetilde{q}_h)\big)(x_j)  \Big)\Big] \\
 =& \prod_{j=1}^n \mathbb{E}\Big[\exp\Big(\lambda n^{-1}\xi_j\big(F_1^d(q_h) - F_1^d(\widetilde{q}_h)\big)(x_j)  \Big)\Big]  \leq \exp\Big(\tfrac12\sigma^2\lambda^2n^{-2}\sum_{j=1}^n \big|\big(F_1^d(q_h) - F_1^d(\widetilde{q}_h)\big)(x_j)\big|^2\Big).
\end{align*}
Thus, the desired claim holds.
By the approximation property \eqref{eqn:w-Pi-h} of the operator $\widetilde\Pi_h$, and Lemmas \ref{lem:reg-u} and \ref{lem:point-err}, we get
\begin{align}
     \|u_h(q_h)-\widetilde\Pi_hu(q_h)\|_n  & \leq \|u_h(q_h)-u(q_h)\|_n + \|u(q_h)-\widetilde\Pi_hu(q_h)\|_n \nonumber\\
     & \leq ch^2\| q_h\|_{W^{1,4}\II}^{2\zeta} +  ch^2\| q_h\|_{W^{1,4}\II}^{\zeta} \leq ch^2\| q_h\|_{W^{1,4}\II}^{2\zeta}. \label{ineq:preliminary-estimate}  
\end{align}
Similarly, the following estimate holds
\begin{equation}\label{dis-F1-estimate in L2}
\|u_h(q_h)-\widetilde\Pi_hu(q_h)\|_{L^2(\Omega)} \leq ch^2\| q_h\|_{W^{1,4}\II}^{2\zeta}.    
\end{equation}
Then the triangle inequality and the choice $\kappa=4\zeta$ lead to
\begin{align}
    &\diam(\mathcal{A}_h)  = \sup_{q_h,\widetilde{q}_h\in\mathcal{A}_h}\widehat{\rbd}(q_h,\widetilde{q}_h) \leq 2\sigma n^{-\frac12}\sup_{q_h\in\mathcal{A}_h}\| F_1^d(q_h)\|_n\nonumber\\
     \leq &c\sigma n^{-\frac12}h^2\| q_h\|_{W^{1,4}\II}^{2\zeta}\Big(\gamma^{\frac12}\|q_h\|^{\frac\kappa2}_{W^{1,4}\II}\Big)^{-1} \leq c\sigma n^{-\frac12}h^2\gamma^{-\frac 12}.\label{dis-F1-estimate-diameter}   
\end{align}
Lemmas \ref{lem:max-ineq} and \ref{lem:cover-en-bound-finite}, the estimates \eqref{dis-F1-estimate in L2} and \eqref{dis-F1-estimate-diameter}, the elementary inequality $\log(1+x) \leq x$ over $\mathbb{R}_+$, the estimate $\mathrm{dim}(X_h + W_h)=O(h^{-d})$, and the choice $\gamma^{\frac12 + \frac d8} = O(\sigma n^{-\frac12}\rho_0^s)$ give
\begin{align*}
 &\quad \Big\|\,\sup_{q_h\in \mathcal{A}_h}|(\bm{\xi}, F_1^d(q_h) )_n|\,\Big\|_{\psi_2} \leq c\int_0^{c\sigma n^{-\frac12}h^2\gamma^{-\frac12}}\sqrt{\log N\big(\tfrac{\epsilon}{2}, \mathcal{A}_h, \widehat{\rbd}\big)} \ {\rm d}\epsilon
    \\& \leq c\int_0^{c\sigma n^{-\frac12}h^2\gamma^{-\frac12}}\sqrt{\log N\big(\tfrac{\epsilon}{c\sigma n^{-\frac12}}, X_h + W_h, \|\cdot\|_{L^2\II}\big)} \ {\rm d}\epsilon \\
    &\leq c\int_0^{c\sigma n^{-\frac12}h^2\gamma^{-\frac12}} \dim(X_h + W_h)^{\frac12}\bigg(\log\Big( 1+ c\sigma n^{-\frac12}\sup_{q_h\in \mathcal{A}_h}\|F_1^d(q_h)\|_{L^2\II}\epsilon^{-1}\Big)\bigg)^{\frac12}\ {\rm d}\epsilon \\
    &\leq c\int_0^{c\sigma n^{-\frac12}h^2\gamma^{-\frac12}} \dim(X_h + W_h)^{\frac12}\Big(c\sigma n^{-\frac12}h^2\gamma^{-\frac12}\epsilon^{-1}\Big)^{\frac12}\ {\rm d}\epsilon \\& \leq c\sigma n^{-\frac12}h^{-\frac{d}{2}}h^2\gamma^{-\frac12} \leq c\gamma^{\frac{d}{8}}h^{2-\frac d2}\rho_0^{-s}.
\end{align*}
This implies the following two estimates
\begin{equation*}
\Big\|\|u_h(q_h^*)-u_h(\Pi_hq^\dagger)\|_n \Big\|_{\psi_2} \leq c\gamma^{\frac d8}h^{4(\frac12-\frac d8)}\rho_0^{-s} \quad\mbox{and}\quad \Big\|\| q_h^*\|_{W^{1,4}\II}\Big\|_{\psi_2} \leq c (\gamma^{-\frac12+\frac d8}h^{4(\frac12-\frac d8)}\rho_0^{-s})^{\frac 2\kappa}.
\end{equation*}
\textbf{Step 2. Bound $F_3$.} By the definition of the event $F_3$, the choice $\kappa = 4\zeta$ and the estimate 
\eqref{W14-stab-Pih}, there holds
\begin{equation}\label{dis-F3-estimate}
F_3 = 2\left(\bm\xi, \frac{\widetilde{\Pi}_hu(\Pi_hq^\dagger) - u_h(\Pi_hq^\dagger)}{\gamma^{\frac12}\|\Pi_hq^\dagger\|_{W^{1,4}\II}^{\frac{\kappa}{2}}}\right)_n\gamma^{\frac12}\|\Pi_hq^\dagger\|_{W^{1,4}\II}^{\frac{\kappa}{2}} \leq c\sup_{q_h\in\mathcal{A}_h}|\big(\bm\xi, F_3^d(q_h)\big)_n| \gamma^{\frac12}\rho_0^{2\zeta},
\end{equation}
with the auxiliary function $F_3^d(q_h)$ given by
\begin{equation*}
F_3^d(q_h):=\frac{\widetilde{\Pi}_hu(q_h) - u_h(q_h)}{\gamma^{\frac12}\|q_h\|_{W^{1,4}\II}^{\frac{\kappa}{2}}}\in X_h+W_h.   
\end{equation*}
Repeating the argument in Step 1 shows that $\{\big(\bm{\xi}, F_3^d(q_h) \big)_n, q_h\in\mathcal{A}_h \}$ is a sub-Gaussian random process equipped with the semi-metric
\begin{equation*}
    \widehat{\rbd}(q_h,\widetilde{q}_h):= \sigma n^{-\frac12}\| F_3^d(q_h) -  F_3^d(\widetilde{q}_h)\|_n, \quad \forall q_h,\widetilde{q}_h \in \mathcal{A}_h.
\end{equation*}
Thus we arrive at
\begin{align*}
\Big\|\,\sup_{q_h\in \mathcal{A}_h}|(\bm{\xi}, F_3^d(q_h) )_n|\,\Big\|_{\psi_2} \leq c\sigma n^{-\frac12}h^{-\frac{d}{2}}h^2\gamma^{-\frac12} \leq c\gamma^{\frac{d}{8}}h^{2-\frac d2}\rho_0^{-s}.
\end{align*}
This estimate,  the inequality \eqref{dis-F3-estimate} and the choice $\kappa =4\zeta$ imply
\begin{equation*}
\Big\|F_3\Big\|_{\psi_2} \leq c\Big\|\sup_{q_h\in\mathcal{A}_h}|\big(\bm\xi, F_3^d(q_h)\big)_n|\Big\|_{\psi_2} \gamma^{\frac12}\rho_0^{2\zeta} \leq c\gamma^{\frac12 + \frac{d}{8}}h^{2-\frac d2}\rho_0^{2\zeta-s}.   
\end{equation*}
Consequently, we deduce
\begin{align*}
\Big\|\|u_h(q_h^*)-u_h(\Pi_hq^\dagger)\|_n\Big\|_{\psi_2} &\leq c(\gamma^{\frac12 + \frac{d}{8}}h^{2-\frac d2}\rho_0^{2\zeta-s})^{\frac12},\\
\Big\|\| q_h^*\|_{W^{1,4}\II}\Big\|_{\psi_2} &\leq  c(\gamma^{-\frac12 + \frac{d}{8}}h^{2-\frac d2}\rho_0^{2\zeta-s})^{\frac 1\kappa}.  
\end{align*}
\textbf{Step 3. Bound $F_2$.} By the triangle inequality and the estimate 
\eqref{ineq:preliminary-estimate}, there holds
\begin{equation}\label{ineq:boundF2-pre-estimate}
\begin{split}
&\|\widetilde{\Pi}_hu(q_h^*) - \widetilde{\Pi}_hu(\Pi_hq^\dag)\|_{n}\\
 \leq & \|u_h(q_h^*) - u_h(\Pi_hq^\dag)\|_{n}+\|\widetilde{\Pi}_hu(q_h^*)-u_h(q_h^*)\|_n + \|\widetilde{\Pi}_hu(\Pi_hq^\dag)-u_h(\Pi_hq^\dag)\|_n\\
 \leq  & \|u_h(q_h^*) - u_h(\Pi_hq^\dag)\|_{n}+ch^2\Big(\|q_h^*\|^{2\zeta}_{W^{1,4}(\Omega)} + \|\Pi_hq^\dag\|^{2\zeta}_{W^{1,4}(\Omega)}\Big).
    \end{split}
\end{equation}
Since $\widetilde{\Pi}_hu(q_h^*) - \widetilde{\Pi}_hu(\Pi_hq^\dag)\in W_h\subset H^2\II$, Lemmas \ref{Geer-entropy} and \ref{lem:reg-u} and the $H^2\II$ stability \eqref{eqn:w-Pi-h} of the operator $\widetilde{\Pi}_h$ give
\begin{align}
  &\big(\bm\xi,\widetilde{\Pi}_hu(q_h^*) - \widetilde{\Pi}_hu(\Pi_hq^\dag)\big)_n \\
  \leq &\mathcal{O}_{p}\left(\sigma n^{-\frac12} \|\widetilde{\Pi}_hu(q_h^*) - \widetilde{\Pi}_hu(\Pi_hq^\dag)\|_{n}^{1- \frac d4}\|\widetilde{\Pi}_hu(q_h^*) - \widetilde{\Pi}_hu(\Pi_hq^\dag)\|_{H^2(\Omega)}^{\frac d4}\right) \nonumber\\
   \leq &\mathcal{O}_{p}\left(\sigma n^{-\frac12}\right) \big[\|u_h(q_h^*) - u_h(\Pi_hq^\dag)\|_{n}+ch^2(\|q_h^*\|^{2\zeta}_{W^{1,4}(\Omega)} + \|\Pi_hq^\dag\|^{2\zeta}_{W^{1,4}(\Omega)})\big]^{1- \frac d4} \nonumber\\&\times \big[\big(\|q_h^*\|_{W^{1,4}\II}+\|\Pi_hq^\dag\|_{W^{1,4}\II}\big)\|q_h^*\|_{W^{1,4}\II}^\zeta\|\Pi_hq^\dag\|_{W^{1,4}\II}^\zeta\big]^{\frac d4}.\label{eqn:est-al0}
\end{align}
Next we analyze the following two subcases separately: (i) $\|q_h^*\|_{W^{1,4}\II} \leq \|\Pi_hq^\dag\|_{W^{1,4}\II}$ and (ii) $\|q_h^*\|_{W^{1,4}\II} > \|\Pi_hq^\dag\|_{W^{1,4}\II}$.
First consider the event (i). By the stability estimate \eqref{W14-stab-Pih}, $$\|q_h^*\|_{W^{1,4}\II} \leq \|\Pi_hq^\dag\|_{W^{1,4}\II}\leq c\rho_0.$$ 
Moreover, from the estimates \eqref{W14-stab-Pih} and \eqref{eqn:est-al0}, we obtain 
\begin{align*}
    & \|u_h(q_h^*) - u_h(\Pi_hq^\dag)\|_{n}^2  \\  \leq& \mathcal{O}_{p}\left(\sigma n^{-\frac12}\right) \Big(\|u_h(q_h^*) - u_h(\Pi_hq^\dag)\|_{n}^{1- \frac d4}\|\Pi_hq^\dag\|_{W^{1,4}\II}^{\frac{(2\zeta+1)d}{4}}+h^{4(\frac12-\frac d8)}\|\Pi_hq^\dag\|_{W^{1,4}\II}^{2\zeta+1}\Big) \\ \leq& 
    \mathcal{O}_{p}\left(\sigma n^{-\frac12}\right) \Big(\|u_h(q_h^*) - u_h(\Pi_hq^\dag)\|_{n}^{1- \frac d4}\rho_0^{\frac{(2\zeta+1)d}{4}}+h^{4(\frac12-\frac d8)}\rho_0^{2\zeta+1}\Big).
\end{align*}
This implies that one of the follow two events must hold
\begin{equation*}
\left\{\begin{aligned}
\|u_h(q_h^*)-u_h(\Pi_hq^\dag)\|_n^{1+\frac d4} &\leq \mathcal{O}_{p}\Big(\sigma n^{-\frac12}\Big)\rho_0^{\frac{(2\zeta+1)d}{4}}, \\
\|u_h(q_h^*)-u_h(\Pi_hq^\dag)\|^2_n &\leq \mathcal{O}_{p}\Big(\sigma n^{-\frac12}\Big)h^{4(\frac12-\frac d8)}\rho_0^{2\zeta+1}.
\end{aligned}\right.
\end{equation*}
The argument in \eqref{con-event-i} and \eqref{con-event-i-state-approx}, and the choice $\gamma^{\frac12 + \frac d8}= O(\sigma n^{-\frac12}\rho_0^s)$ give
\begin{equation*}
    \Big\|\|u_h(q_h^*)-u_h(\Pi_hq^\dag)\|_n\Big\|_{\psi_2} \leq c [\gamma^{\frac12} \rho_0^{\frac{\kappa}{2}} + (\gamma^{\frac12 + \frac d8}h^{4(\frac12-\frac d8)}\rho_0^{2\zeta+1-s} )^\frac12 ]\ \mbox{ and }\ \Big\|\|q_h^*\|_{W^{1,4}\II}\Big\|_{\psi_2}\leq c\rho_0.
\end{equation*}
Next consider the event (ii). Similarly, by the box constraint of $\mathcal{A}$, we get
\begin{align*}
    & \|u_h(q_h^*) - u_h(\Pi_hq^\dag)\|_{n}^2 + \gamma \|q_h^*\|_{W^{1,4}\II}^\kappa \\  \leq & \mathcal{O}_{p}\left(\sigma n^{-\frac12}\right) \Big(\|u_h(q_h^*) - u_h(\Pi_hq^\dag)\|_{n}^{1- \frac d4}\|q_h^*\|_{W^{1,4}\II}^{\frac{(2\zeta+1)d}{4}}+h^{4(\frac12-\frac d8)}\|q_h^*\|_{W^{1,4}\II}^{2\zeta(1-\frac{d}{4})+(2\zeta+1)\frac{d}{4}}\Big) \\
\leq &\mathcal{O}_{p}\left(\sigma n^{-\frac12}\right) \Big(\|u_h(q_h^*) - u_h(\Pi_hq^\dag)\|_{n}^{1- \frac d4}\|q_h^*\|_{W^{1,4}\II}^{\frac{(2\zeta+1)d}{4}}+h^{4(\frac12-\frac d8)}\|q_h^*\|_{W^{1,4}\II}^{2\zeta+1}\Big).
\end{align*}
Thus one of the follow two events must hold
\begin{equation}\label{alternative-estimate}
\left\{\begin{aligned}
\|u_h(q_h^*) - u_h(\Pi_hq^\dag)\|_{n}^2 + \gamma \|q_h^*\|_{W^{1,4}\II}^\kappa &\leq \mathcal{O}_{p}\Big(\sigma n^{-\frac12}\Big)\|u_h(q_h^*) - u_h(\Pi_hq^\dag)\|_{n}^{1- \frac d4}\|q_h^*\|_{W^{1,4}\II}^{\frac{(2\zeta+1)d}{4}}, \\
\|u_h(q_h^*) - u_h(\Pi_hq^\dag)\|_{n}^2 + \gamma \|q_h^*\|_{W^{1,4}\II}^\kappa &\leq \mathcal{O}_{p}\Big(\sigma n^{-\frac12}\Big)h^{4(\frac12-\frac d8)}\|q_h^*\|_{W^{1,4}\II}^{2\zeta+1}.
\end{aligned}\right.
\end{equation}
From the first inequality in \eqref{alternative-estimate}, we obtain
\begin{align*}
\left\{\begin{aligned}    \gamma\|q_h^*\|_{W^{1,4}\II}^{\kappa - \frac{(2\zeta+1)d}{4}} &\leq \mathcal{O}_{p}(\sigma n^{-\frac12})\|u_h(q_h^*) - u_h(\Pi_hq^\dag)\|_{n}^{1- \frac d4}, \\
    \|u_h(q_h^*) - u_h(\Pi_hq^\dag)\|_{n}^{1+\frac d4} &\leq \mathcal{O}_{p}\left(\sigma n^{-\frac12}\right) \|q_h^*\|_{W^{1,4}\II}^{\frac{(2\zeta+1)d}{4}}.
\end{aligned}\right.
\end{align*}
Similarly to \eqref{con-event-ii-alternative-ineq}, solving system of inequalities  and using the choice $\gamma^{\frac12 + \frac d8}= O(\sigma n^{-\frac12}\rho_0^s)$ yield
\begin{equation*}
\Big\|\|u_h(q_h^*) - u_h(\Pi_hq^\dag)\|_n\Big\|_{\psi_2} \leq c\gamma^{\frac12}\rho_0^{\frac{\kappa}{2}}\quad\mbox{and}\quad \Big\|\|q_h^*\|_{W^{1,4}\II} \Big\|_{\psi_2} \leq c\rho_0.
\end{equation*}
Last, solving the second inequality in \eqref{alternative-estimate} gives
\begin{align*}
\Big\|\|q_h^*\|_{W^{1,4}\II}\Big\|_{\psi_2}&\leq c(\gamma^{-\frac12+\frac d8}h^{4(\frac12-\frac d8)}\rho_0^{-s})^{\frac{1}{\kappa-2\zeta-1}}, \\  
\Big\|\|u_h(q_h^*) - u_h(\Pi_hq^\dag)\|_{n}\Big\|_{\psi_2} &\leq c\gamma^{\frac12}(\gamma^{-\frac12+\frac d8}h^{4(\frac12-\frac d8)}\rho_0^{-s})^{\frac{\kappa}{2(\kappa-2\zeta-1)}}.
\end{align*}
Combining the preceding estimates gives the first assertion. Then choosing $h = O(\gamma^\frac14)$ gives the second estimate.
\end{proof}

\begin{remark}
Lemma \ref{lem:approx-sol-uppbound-qh} is a discrete analogue of Lemma \ref{lem:apprx-u-bound-q}, but the exponent on $\rho_0$ is larger. This is due to the use of the duality argument in establishing the $L^2(\Omega)$ error bound on the state approximation $u_h(q)$; see the proof of Lemma \ref{lem:point-err}.
\end{remark}

Now we can state the proof of Theorem \ref{thm:err-conduc}.

\begin{proof}
    By the estimate \eqref{ineq:proba-distrib-estimate}, and Lemmas  \ref{lem:point-err} and \ref{lem:approx-sol-uppbound-qh}, the following two estimates hold simultaneously with probability at least $1-2\tau$
\begin{align*}
\|u_h(q_h^*)-u^\dag\|_n&\leq \|u_h(q_h^*)-u_h(\Pi_hq^\dag)\|_n + \|u_h(\Pi_hq^\dag)-u^\dag\|_n \leq c\big(\eta\ell_\tau + h^2\big),\\ 
\|q_h^*\|_{W^{1,4}\II} &\leq c\omega\ell_\tau .
\end{align*} 
Using the weak formulations of $u^\dag$ and $u_h(q_h^*)$, we obtain for any $\varphi\in H_0^1\II$,
\begin{align*}
    ((q^\dag - q_h^*)\nabla u^\dag, \nabla \varphi) & = ((q^\dag - q_h^*)\nabla u^\dag, \nabla (\varphi - P_h\varphi)) + ((q^\dag - q_h^*)\nabla u^\dag, \nabla P_h\varphi) \\
    & = -(\nabla\!\cdot\!\big[(q^\dag - q_h^*)\nabla u^\dag\big], \varphi - P_h\varphi) + (q_h^*\nabla (u_h(q_h^*) - u^\dag), \nabla P_h\varphi) = :{\rm I} + {\rm II}.
\end{align*}
Let $\varphi = \frac{q^\dag - q_h^*}{q^\dag} u^\dag$. Then repeating the argument of the estimate \eqref{estimate-grad-varphi} and Poincar\'{e} inequality give
\begin{equation*}
    \|\varphi\|_{H^1\II}\leq c\|\nabla \varphi\|_{L^2\II} \leq c \omega\ell_\tau.
\end{equation*}
Hence $\varphi\in H_0^1\II$. Next we bound the terms ${\rm I}$ and ${\rm II}$ separately. By the triangle inequality, and Lemmas \ref{lem:Sob-embed} and \ref{lem:reg-u},
\begin{align*}
     \|\nabla\!\cdot\!\big[(q^\dag - q_h^*)\nabla u^\dag\big]\|_{L^2\II}& \leq \|\nabla (q^\dag - q_h^*)\cdot\nabla u^\dag\|_{L^2\II} + \|(q^\dag - q_h^*)\Delta u^\dag\|_{L^2\II} \\
    & \leq \|\nabla (q^\dag - q_h^*)\|_{L^4\II}\|\nabla u^\dag\|_{L^4\II} + \|q^\dag - q_h^*\|_{L^\infty\II}\|\Delta u^\dag\|_{L^2\II} \\
    & \leq c\|q^\dag - q_h^*\|_{W^{1,4}\II} \|u^\dag\|_{H^2\II} \leq c\|q^\dag - q_h^*\|_{W^{1,4}\II}\|q^\dag\|_{W^{1,4}\II}^\zeta.
\end{align*}
Thus the Cauchy-Schwarz inequality leads to
\begin{align*}
    & |{\rm I}| \leq ch\|\nabla\!\cdot\!\big[(q^\dag - q_h^*)\nabla u^\dag\big]\|_{L^2\II}\|\nabla \varphi\|_{L^2\II}   \\
    \leq&  ch\|q^\dag - q_h^*\|_{W^{1,4}\II}\|q^\dag\|_{W^{1,4}\II}^\zeta\|\nabla \varphi\|_{L^2\II} \leq ch(\omega \ell_\tau)^{\zeta+2}.
\end{align*}
Next  we split the term ${\rm II}$ into two parts
\begin{equation*}
    {\rm II} = (q_h^*\nabla (u_h(q_h^*) - u(q_h^*)), \nabla P_h\varphi) + (q_h^*\nabla (u(q_h^*) - u^\dag), \nabla P_h\varphi) = : {\rm II}_1 + {\rm II}_2.
\end{equation*}
Using  C\'{e}a's lemma \cite[Chapter 5]{BrennerScott:book2008} and Lemma \ref{lem:reg-u}, there holds
\begin{equation*}
\|\nabla (u_h(q_h^*) - u(q_h^*))\|_{L^2\II} \leq ch \| u(q_h^*)\|_{H^2\II}\leq ch\|q_h^*\|_{W^{1,4}\II}^\zeta \leq  ch(\omega\ell_\tau)^\zeta.   
\end{equation*}
By H\"{o}lder's inequality, the box constraint of $\mathcal{A}$ and the $H^1\II$ stability \eqref{ineq:P_h-approx} of the operator $P_h$, we get
\begin{align*}
    |{\rm II}_1| &\leq c\|\nabla (u_h(q_h^*) - u(q_h^*))\|_{L^2\II} \|\nabla P_h\varphi\|_{L^2\II} \leq ch(\omega\ell_\tau)^\zeta \|\nabla\varphi\|_{L^2\II} \leq ch(\omega \ell_\tau)^{\zeta+1}.
\end{align*}
Next by the Gagliardo–Nirenberg interpolation inequality, and Lemmas \ref{lem:connect-seminorm-stdnorm} and \ref{lem:point-err}, we obtain
\begin{align*}
& \|\nabla (u(q^*_h) - u^\dag)\|_{L^2\II}  \leq c\|u(q^*_h) - u^\dag\|_{L^2\II}^{\frac12}\|u(q^*_h) - u^\dag\|_{H^2\II}^{\frac12} \\
 \leq &c\big[\|u(q^*_h) - u^\dag\|_n + n^{-\frac2d}\|u(q^*_h) - u^\dag\|_{H^2\II}\big]^{\frac12}\|u(q^*_h) - u^\dag\|_{H^2\II}^{\frac12} \\
 \leq &c\big[\|u_h(q^*_h) - u^\dag\|_n + h^2\|q_h^*\|_{W^{1,4}\II}^{2\zeta} + n^{-\frac2d}\|u(q^*_h) - u^\dag\|_{H^2\II}\big]^{\frac12}\|u(q^*_h) - u^\dag\|_{H^2\II}^{\frac12} \\
\leq &c(\eta + h^2 + n^{-\frac2d})^{\frac12}(\|q_h^*\|_{W^{1,4}\II}^{2\zeta+1} + \|q^\dag\|_{W^{1,4}\II}^{2\zeta+1}) \leq  c(\eta + h^2 + n^{-\frac2d})^{\frac12}(\omega \ell_\tau)^{2\zeta+1}.
\end{align*}
By H\"{o}lder's inequality, the box constraint of $\mathcal{A}$ and the $H^1\II$ stability of $P_h$, we arrive at 
\begin{equation*}
    {\rm II_2} \leq c\|\nabla (u(q^*_h) - u^\dag)\|_{L^2\II}\|\nabla \varphi\|_{L^2\II} \leq c(\eta + h^2 + n^{-\frac2d})^{\frac12}(\omega \ell_\tau)^{2\zeta+2}.
\end{equation*}
By combining the preceding estimates with the following identity
\begin{equation*}
((q^\dag - q^*_h)\nabla u^\dag,\nabla\varphi) = \frac12\int_{\Omega}\Big(\frac{q^\dag - q^*_h}{q^\dag}\Big)^2(fu^\dag + q^\dag|\nabla u^\dag|^2)\ {\rm d}x,    
\end{equation*}
we arrive at the first estimate. The rest of the proof follows identically as Theorem \ref{thm:con-err-conduc}.
\end{proof}

\section{Numerical experiments and discussions}
\label{sec:numer}

Now we present numerical experiments to complement the theoretical analysis. The lower and upper bounds of the admissible set $\mathcal{A}$ are taken to be $c_0 = 1.0$ and $c_1 = 3.0$, respectively. The domain $\Omega$ is the unit square $(0,1)^2$. The observation points $(x_i)_{i=1}^n$ are uniformly distributed over the domain $\Omega$. We divide the domain $\Omega$ into the shape regular finite element mesh $\mathcal{T}_h$, and obtain the reference solution $u^\dag$ using a finer grid with a mesh size $h=\frac{1}{100}$. The noisy data $\boldsymbol{m}$ is generated by
\begin{equation*}
	m_i = u^\dagger(x_i) + \sigma\|u^\dagger\|_{L^\infty(\Omega)}\xi_i,\quad i=1,2,\dots,n,
\end{equation*}
where the random variables $\bm \xi = (\xi_i)_{i=1}^n$ follow the standard normal distribution and $\sigma$ denotes the noise strength. To solve problem \eqref{eqn:dis-optim-dis}-\eqref{eqn:dis-weak-conduc}, we employ the steepest descent method. Now we describe the algorithm
for problem \eqref{eqn:conti-optim-prob-conduc}--\eqref{eqn:conti-weak-conduc}. To compute the gradient $J'_\gamma(q)$ of $J_\gamma(q)$, we employ the adjoint state: find $v(q)\in H_0^1\II$ such that
\begin{equation}\label{eqn:adj}
(q\nabla v(q),\nabla\varphi) = 2(u(q)-\bm \xi,\varphi)_n, \quad\forall \varphi\in H^1_0\II.
\end{equation}
Then the gradient $g\equiv J_\gamma'(q)\in W^{1,4}_0\II$ is given by the solution to the following $4$-Laplace equation
\begin{equation}\label{eqn:grad}
	 -\nabla\!\cdot\!\big(|\nabla g|^2\nabla g\big)= 8\gamma \|q\|_{W^{1,4}\II}^4\big(\!-\!\nabla\!\cdot\!( |\nabla q|^2\nabla q) + q^3\big) \!-\! \nabla u(q)\!\cdot\! \nabla v(q),
\end{equation}
Due to the nonlinearity of \eqref{eqn:grad}, we employ the Newton iteration with the initial guess $g^0$ given by the solution of the Poisson equation (with a zero Dirichlet boundary) below
\begin{equation*}
    -\Delta g^0 = 8\gamma \|q\|_{W^{1,4}\II}^4\big(\!-\!\nabla\!\cdot\!( |\nabla q|^2\nabla q) + q^3\big) \!-\! \nabla u(q)\!\cdot\! \nabla v(q).
\end{equation*}
To examine the convergence behavior of the discrete approximation $q_h^*$, we employ the errors
\begin{equation*}
e_q:= \|q^\dag - q_h^*\|_{L^2\II}/\|q^\dag\|_{L^2\II}\quad\mbox{and}\quad e_u:=\|u^\dag - u_h(q_h^*)\|_n.
\end{equation*}
Throughout we choose the regularization strength $\gamma$ according to the theoretical guideline
\begin{equation}\label{eqn:opt-gamma}
    \gamma^{\frac34} = O(\sigma n^{-\frac12} \rho_0^{-\frac72}), \ \mbox{with}\ \rho_0= \|q^\dag\|_{W^{1,4}\II} + \sigma n^{-\frac12}.
\end{equation}
\begin{example}\label{exam:opt-gamma}
Let $q^\dagger=1.0+\frac12\sin(\pi x)\sin(\pi y)$ and $f\equiv1$. Fix the number $n$ of sampling points at $201^2$, and a mesh size $h=\frac{1}{25}$ small enough to neglect the FEM error. Consider two noise levels: $\sigma=5.0\%$ and $\sigma=1.0\%$. Then we solve problem \eqref{eqn:dis-optim-dis}-\eqref{eqn:dis-weak-conduc} with $\gamma\in\{10^{-7},10^{-7.5},10^{-8.0},10^{-8.5},10^{-9},10^{-9.5},10^{-10}\}$.
\end{example}

Fig. \ref{fig:opt-gamma} shows the errors $e_q$ and $e_u$ at different $\gamma$ values. The optimal parameters are $10^{-8.5}$ and $10^{-9}$ for $\sigma=5.0\%$ and $\sigma=1.0\%$, respectively. Both values are close to the automatically determined $3.32$e-9 ($\sigma=5.0\%$) and $3.89$e-10 ($\sigma=1.0\%$) by the \textit{a priori} choice \eqref{eqn:opt-gamma}. This observation confirms the effectiveness of the \textit{a priori} choice \eqref{eqn:opt-gamma}.

\begin{figure}[hbt!]
\centering
\setlength{\tabcolsep}{0pt}
\begin{tabular}{cc}
\includegraphics[width=0.48\textwidth]{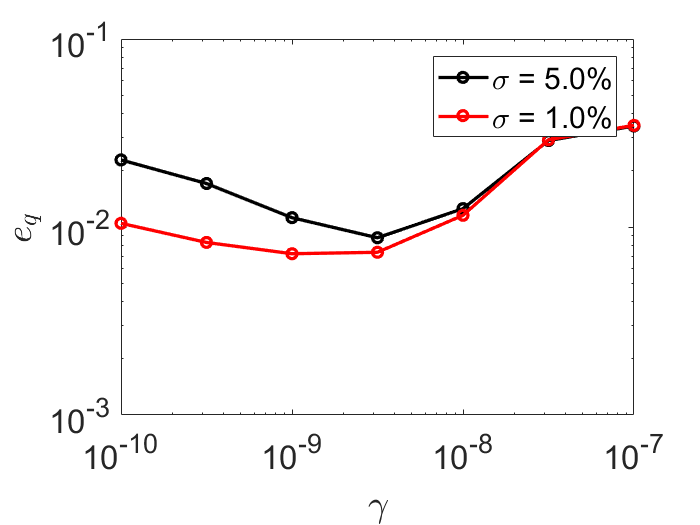} & \
\includegraphics[width=0.48\textwidth]{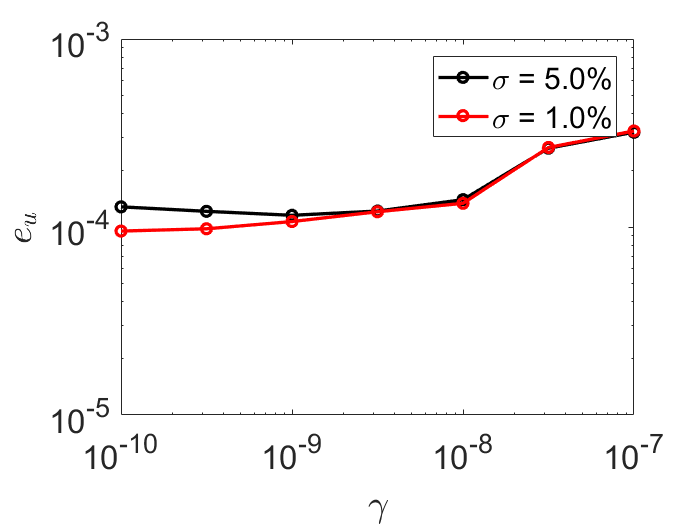} 
\end{tabular}
\caption{The numerical results on the optimal $\gamma^*$ at two noise levels: the error $e_q$ versus the parameter $\gamma$ {\rm(}left{\rm)}; the error $e_u$ versus the parameter $\gamma$ {\rm(}right{\rm)}.}
\label{fig:opt-gamma}
\end{figure}

\begin{example}\label{exam:rate}
Let $\gamma$ be given by \eqref{eqn:opt-gamma} and $h = O(\gamma^\frac14)$. Consider two noise levels: $\sigma=5.0\%$ and $\sigma=1.0\%$ with the number $n$ of sampling points varying in the set $\{51^2,101^2,201^2,401^2\}$. Consider the following two settings:
\begin{enumerate}
    \item[(a)] $q^\dagger=1.0+\frac12\sin(\pi x)\sin(\pi y)$ and $f\equiv1$;
    \item[(b)] $q^\dagger=1.0+\frac12\sin(\pi x)\sin(\pi y)w$ and $f\equiv1$;
\end{enumerate}
with the function $w= e^{-10((x-\frac14)^2 + (y-\frac14)^2)}+e^{-10((x-\frac34)^2 + (y-\frac34)^2)}$.
\end{example}
The exact conductivity $q^\dag$ in (a) involves one single bump but that in (b) involves multiple bumps. For both noise levels ($\sigma=5.0\%$ and $\sigma=1.0\%$), the metrics $e_q$ and $e_u$ in Table \ref{tab:rate} exhibit a steady convergence to zero as the number $n$ of sampling points tends to infinity. Fig. \ref{fig:rate} illustrates the numerical approximation $q_h^*$ for the two conductivity profiles, with the sampling points $n$ taken at $201^2$ (middle column) and $401^2$ (last column). In both cases, the pointwise errors concentrate near the bumps of the state $u^\dag$. Notably, the numerical results in case (b) captures well the multiple bump contours.

\begin{table}[hbt!]
\caption{Numerical results for Example \ref{exam:rate}}
\centering
\begin{threeparttable}
\begin{minipage}[b]{0.48\textwidth}
\centering
\subcaption*{(a) with $\sigma=5.0\%$.}
\begin{tabular}{c|ccccc}
\toprule
$n$ & $51^2$ & $101^2$ & $201^2$ & $401^2$ & trend \\ \midrule
$\gamma$ & 4.45e-7 & 1.79e-7 & 7.16e-8 & 2.85e-8 &  \\ \midrule
$e_q$ & 6.68e-2 & 3.77e-2 & 3.23e-2 & 2.40e-2 & $\searrow$ \\
$e_u$ & 7.16e-4 & 3.85e-4 & 2.97e-4 & 2.17e-4 & $\searrow$ \\
\bottomrule
\end{tabular}
\end{minipage}
\hfill
\begin{minipage}[b]{0.48\textwidth}
\centering
\subcaption*{(a) with $\sigma=1.0\%$.}
\begin{tabular}{c|ccccc}
\toprule
$n$ & $51^2$ & $101^2$ & $201^2$ & $401^2$ & trend \\ \midrule
$\gamma$ & 5.21e-8 & 2.10e-8 & 8.37e-9 &  3.33e-9 &  \\ \midrule
$e_q$ & 3.13e-2  & 2.28e-2 & 1.61e-2 &  9.77e-3 & $\searrow$ \\
$e_u$ & 2.75e-4   & 2.18e-4 & 1.73e-4 &  1.32e-4 & $\searrow$ \\ \bottomrule
\end{tabular}
\end{minipage}

\begin{minipage}[b]{0.48\textwidth}
\centering
\subcaption*{(b) with $\sigma=5.0\%$.}
\begin{tabular}{c|ccccc}
\toprule
$n$ & $51^2$ & $101^2$ & $201^2$ & $401^2$ & trend \\ \midrule
$\gamma$ & 4.42e-8 & 1.78e-8 & 7.11e-9 & 2.83e-9 &  \\
\midrule
$e_q$ & 4.29e-2& 3.24e-2 & 2.13e-2 & 1.39e-2 & $\searrow$ \\
$e_u$ & 2.30e-4 & 1.73e-4 & 1.06e-4 & 8.35e-5 & $\searrow$ \\
\bottomrule
\end{tabular}
\end{minipage}
\hfill
\begin{minipage}[b]{0.48\textwidth}
\centering
\subcaption*{(b) with $\sigma=1.0\%$.}
\vspace{0.5em}
\begin{tabular}{c|ccccc}
\toprule
$n$ & $51^2$ & $101^2$ & $201^2$ & $401^2$ & trend \\ \midrule
$\gamma$ & 5.17e-9 & 2.08e-9 & 8.31e-10 & 3.31e-10  &  \\ \midrule
$e_q$ & 2.07e-2 & 1.57e-2 & 1.23e-2 & 1.02e-2 & $\searrow$ \\
$e_u$ & 1.01e-4 & 8.72e-5 & 7.32e-5 & 6.96e-5 & $\searrow$ \\ \bottomrule
\end{tabular}
\end{minipage}
\end{threeparttable}
\label{tab:rate}
\end{table}

\begin{figure}[hbt!]
\centering
\setlength{\tabcolsep}{0pt}
	\begin{tabular}{ccc}
		\includegraphics[width=0.32\textwidth]{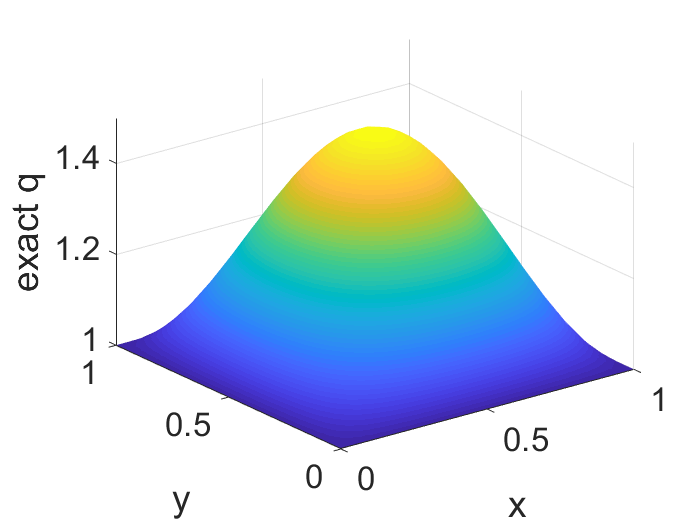} & \includegraphics[width=0.32\textwidth]{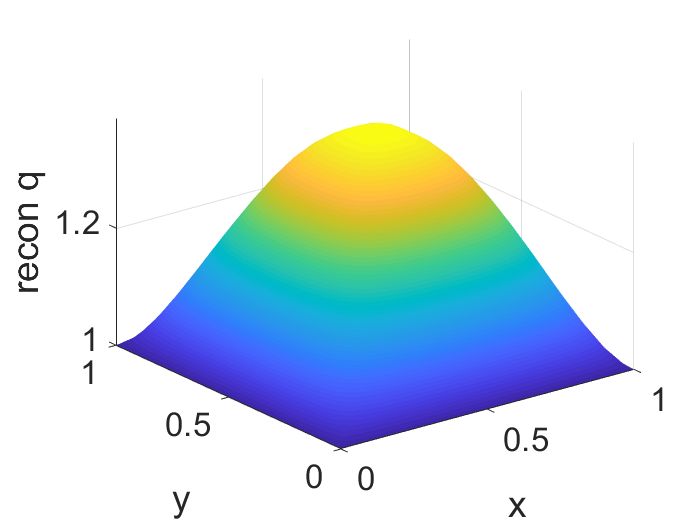} &
		\includegraphics[width=0.32\textwidth]{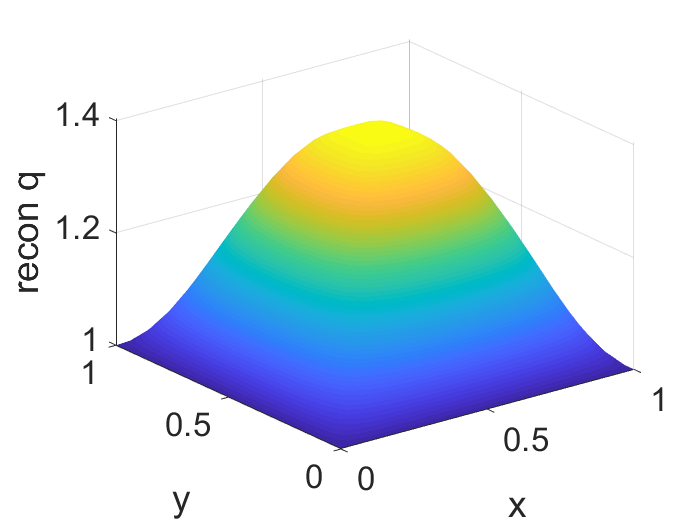} \\
		 (a): exact & $\gamma = 7.16$e-8, $n=201^2$ & $\gamma = 2.85$e-8, $n=401^2$
	\end{tabular}
\begin{tabular}{ccc}
		\includegraphics[width=0.32\textwidth]{Fig/exact.png} & \includegraphics[width=0.32\textwidth]{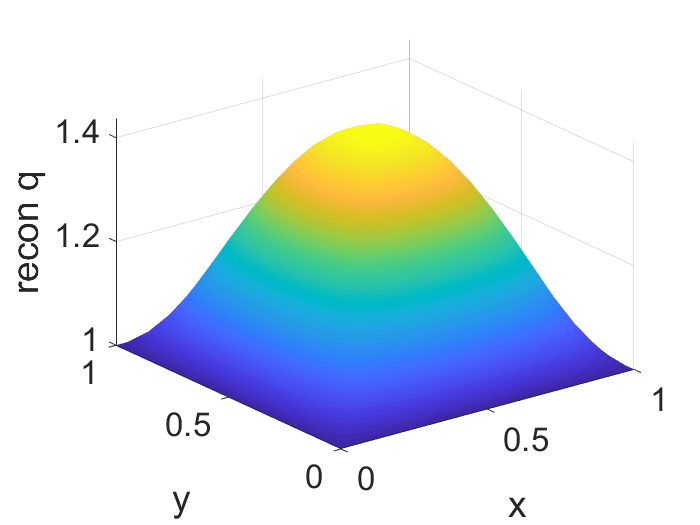} &
		\includegraphics[width=0.32\textwidth]{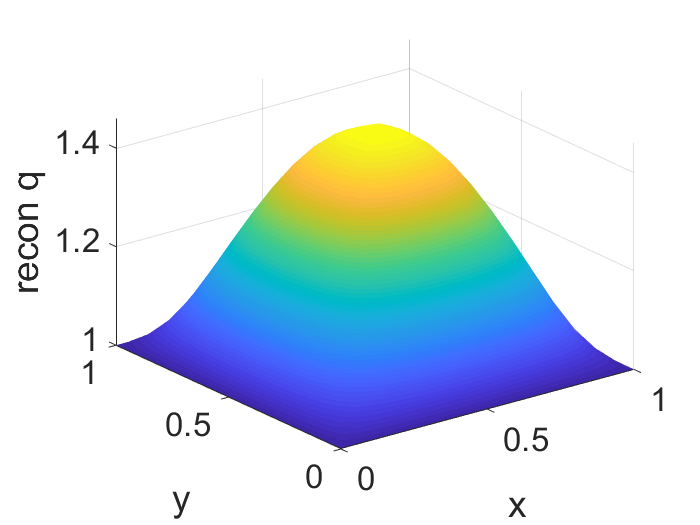} \\
		 (a): exact & $\gamma = 8.37$e-9, $n=201^2$ & $\gamma = 3.33$e-9, $n=401^2$
	\end{tabular}
    \begin{tabular}{ccc}
		\includegraphics[width=0.32\textwidth]{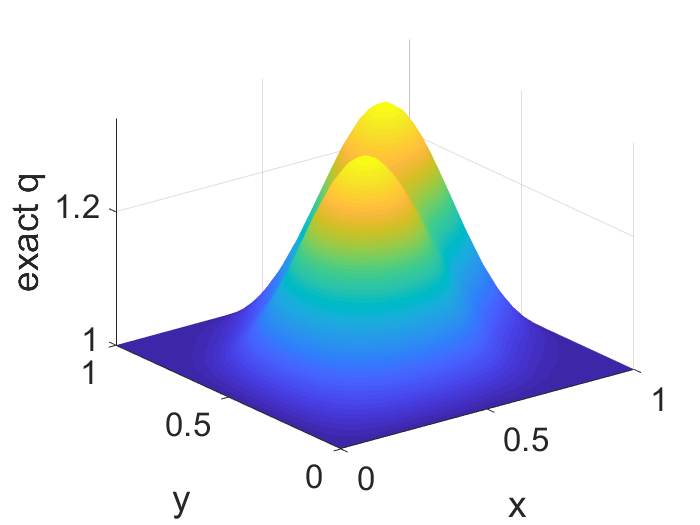} & \includegraphics[width=0.32\textwidth]{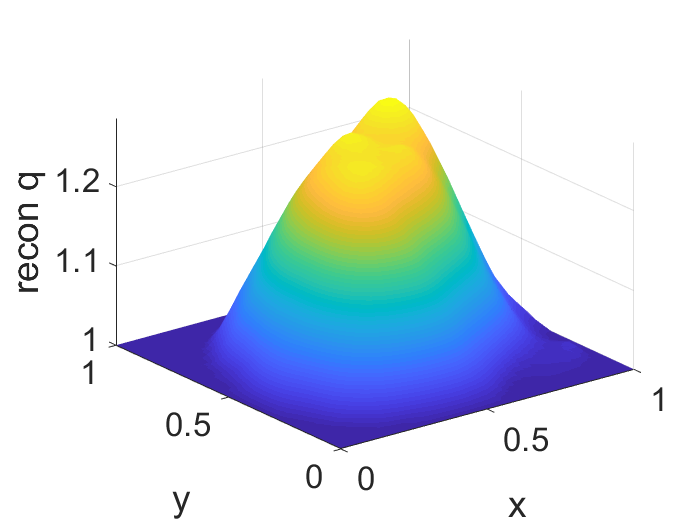} &
		\includegraphics[width=0.32\textwidth]{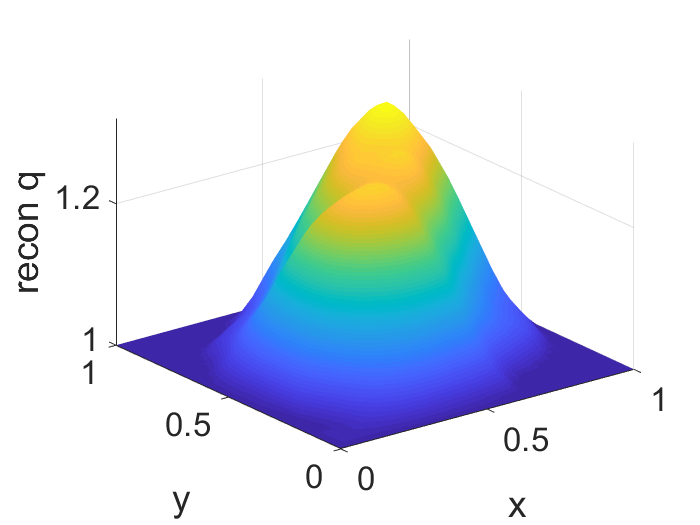} \\
		 (b): exact & $\gamma = 7.11$e-9, $n=201^2$ & $\gamma = 2.83$e-9, $n=401^2$
	\end{tabular}
    \begin{tabular}{ccc}
		\includegraphics[width=0.32\textwidth]{Fig/exact-plus.png} & \includegraphics[width=0.32\textwidth]{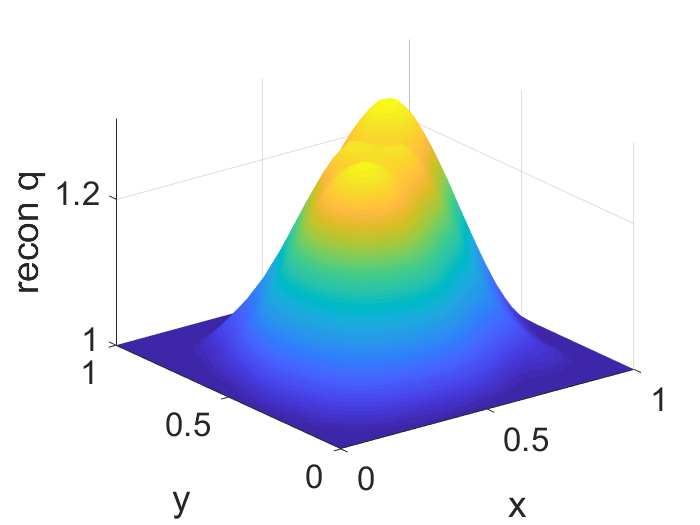} &
		\includegraphics[width=0.32\textwidth]{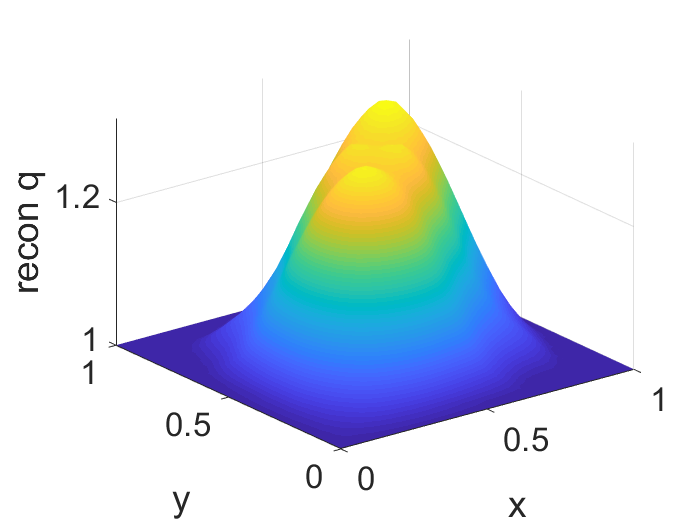} \\
		 (b): exact & $\gamma = 8.31$e-10, $n=201^2$ & $\gamma = 3.31$e-10, $n=401^2$
	\end{tabular}
	\caption{The reconstruction $q_{h}^*$ of Example \ref{exam:rate} for $\sigma=5.0\%$ (top) and $\sigma=1.0\%$ (bottom).}
    \label{fig:rate}
\end{figure}

\appendix

\section{Preliminaries on sub-Gaussian random variables}
\label{sec:pre}
In this appendix, we review basic theory of sub-Gaussian random variables. First we give the definition of sub-Gaussian random variables and useful properties of empirical processes. The notation $\mathbb{E}$ denotes taking expectation with respect to the involved random variables.

\begin{definition}\label{def:sub-Gau}
A random variable $Z$ is sub-Gaussian with parameter $\sigma$ if it satisfies
\begin{equation*}
\mathbb{E}\bigg[\exp(\lambda (Z - \mathbb E[Z]))\bigg] \leq \exp\bigg(\frac12\sigma^2\lambda^2\bigg),\quad \forall\lambda\in\mathbb{R}.
\end{equation*}
The probability distribution function of a sub-Gaussian random variable $Z$ has an exponentially decaying tail:
\begin{equation*}
\mathbb{P}(|Z - \mathbb{E} [Z]|\ge z)\leq 2\exp\bigg(-\frac{z^2}{2\sigma^2}\bigg),\quad\forall z>0.
\end{equation*} 
\end{definition}

Next we define  the Orlicz norm. Let $\psi$ be a monotonically increasing convex function on $\mathbb{R}_+\cup\{0\}$ with $\psi(0)=0$, 
the Orlicz norm $\|Z\|_\psi$ of a random variable $Z$ is given by
\begin{equation*}
	\|Z\|_{\psi} = \inf\biggl\{c>0:\mathbb{E}\Big[\psi\Big(\frac{|Z|}c\Big)\Big]\le 1\biggr\}.
\end{equation*}
Throughout, we choose $\psi_2(t):=\exp(t^2)-1$ for $t\geq0$. Then we have the following statements \cite[pp. 643-644]{chen2018stochastic}.
\begin{lemma}\label{lem:bdd-sub-Gaussian}
	For any sub-Gaussian random variable $Z$ with $\|Z\|_{\psi_2} < \infty$, the following estimate holds:
	\begin{equation}\label{ineq:proba-distrib-estimate}
		\mathbb{P}(|Z|\ge z)\le 2\exp\bigg(-\frac{z^2}{\|Z\|_{\psi_2}^2}\bigg),\quad \forall z>0.
	\end{equation}
       Conversely, if a random variable $Z$ satisfies
	\begin{equation*}
		\mathbb{P}(|Z|>\alpha (1+z))\leq c_1\exp\big(-c_2^2z^2\big), \quad \forall \alpha>0, ~~z\ge 1,
	\end{equation*}
	for some constants $c_1,c_2>0$. Then there exists $c_3>0$ depending on $c_1$ and $c_2$ such that 
	\begin{equation}\label{ineq:orlicz-bound}
		\|Z\|_{\psi_2}\le c_3\alpha.
	\end{equation}
\end{lemma}
 
\begin{definition}
Let $\mathbb{T}$ be a semimetric space with a semimetric $\rbd$. The random process $\{Z_t:t\in \mathbb{T}\}$ is sub-Gaussian if
\begin{equation*}
\mathbb{P}(|Z_s-Z_t|>z)\le 2\,\exp\bigg( -\frac{z^2}{2\,\rbd(s,t)^2} \bigg),\quad \forall s,t\in \mathbb{T}, ~~z>0.
\end{equation*}
\end{definition}
\begin{definition}
  A collection of points $(t_i)_{i=1}^n \subset \mathbb{T}$ is called an
	$\epsilon$-cover of $\big(\mathbb{T},\rbd\big)$ if for any $t\in \mathbb{T}$, there
	exists at least one $i \in \{1,\dots,n\}$ such that $\rbd(t, t_i) \leq \epsilon$. The
	$\epsilon$-covering number $N(\epsilon,\mathbb{T},\rbd)$ is the minimum
	cardinality among all $\epsilon$-cover of $\mathbb{T}$.  
\end{definition}

We use the following maximal inequality frequently \cite[Section 2.2.1]{vanDerVaart:1996}.
\begin{lemma}\label{lem:max-ineq}
	Let $\{Z_t:t\in \mathbb{T}\}$ be a separable sub-Gaussian random process. Then there holds
	\begin{equation*}
		\Big\|\sup_{s,t\in \mathbb{T}}|Z_s-Z_t|\Big\|_{\psi_2}\leq c\int^{\diam\, \mathbb{T}}_0\sqrt{\log N\big(\tfrac12\epsilon, \mathbb{T},\rbd\big)}~{\rm d}\epsilon.
	\end{equation*}
\end{lemma}

The next two lemmas give bounds on the covering number for Sobolev subsets and finite dimensional subsets \cite{birman1967piecewise,geer2000empirical}, respectively.

\begin{lemma}\label{lem:cover-en-bound-Sobolev}
	Let $Q$ be the unit cube in $\mathbb{R}^d$ and $SW^{s,p}(Q)$ be the unit sphere of space $W^{s,p}(Q)$ for $s> 0$ and $p\ge 1$. Then for small $\epsilon>0$, there holds  
	\begin{equation*}
		\log N(\epsilon, SW^{s,p}(Q), \|\cdot\|_{L^q(Q)})\le c\epsilon^{-\frac ds},
	\end{equation*}
	where $ 1\le q\le\infty$ for $sp>d$ and $1\le q <p(1-sp/d)^{-1}$ for $sp\leq d$.
\end{lemma}

\begin{lemma}\label{lem:cover-en-bound-finite}
	Let $G\subset L^2\II$ be a finite dimensional subspace of dimension $N_G=\dim(G)>0$, and $G_R=\{g\in G: \|g\|_{L^2\II}\le R\}$. Then for small $\epsilon>0$, there holds 
	\begin{equation*}
		\log N(\epsilon,G_R,\|\cdot\|_{L^2\II})\leq N_G\log(1+4R\epsilon^{-1}).
	\end{equation*}
\end{lemma}

\bibliographystyle{abbrv}
\bibliography{reference}

\end{document}